\documentclass[final,12pt,3p]{elsarticle}

\usepackage{graphicx}
\usepackage{amssymb}
\usepackage{amsthm}
\usepackage{amsmath}
\usepackage{algorithm}
\usepackage{algorithmic}
\usepackage{tikz}
\usepackage{lineno}
\usepackage{hyperref}

\usepackage{caption}
\usetikzlibrary{shapes}
\usepackage{subcaption}
\usepackage{geometry}
\usepackage{mathrsfs} 
\usepackage{array}
\usepackage{multirow}
\usepackage{booktabs} 
\usepackage{adjustbox}
\usepackage[title]{appendix}

\biboptions{sort&compress}
\hypersetup{
colorlinks=true,
linkcolor=red,
filecolor=green,
urlcolor=green,
anchorcolor=green,
citecolor=cyan,
pdftitle={Overleaf Example},
pdfpagemode=FullScreen,
}
\newtheorem{theorem}{Theorem}
\newtheorem{lemma}{Lemma}

\newtheorem{corollary}{Corollary}
\newtheorem{proposition}{Proposition}

 \usepackage{xcolor}

 \usepackage[figcolor=white]{todonotes}

\begin{document}
	
\begin{tikzpicture}[remember picture,overlay]
\node[anchor=north east,inner sep=20pt] at (current page.north east)
{\includegraphics[scale=0.2]{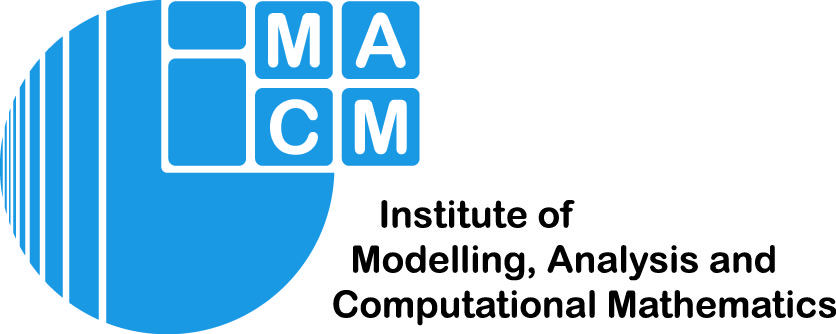}};
\end{tikzpicture}

\begin{frontmatter}

\title{Spatiotemporal SEIQR Epidemic Modeling with Optimal Control\\ for Vaccination, Treatment, and Social Measures}

\author[BUW,MAIS]{Achraf Zinihi}
\ead{a.zinihi@edu.umi.ac.ma} 

\author[BUW]{Matthias Ehrhardt\corref{Corr}}
\cortext[Corr]{Corresponding author}
\ead{ehrhardt@uni-wuppertal.de}

\author[MAIS]{Moulay Rchid Sidi Ammi}
\ead{rachidsidiammi@yahoo.fr}

\address[BUW]{University of Wuppertal, Applied and Computational Mathematics,\\
Gaußstrasse 20, 42119 Wuppertal, Germany}

\address[MAIS]{Department of Mathematics, AMNEA Group, Faculty of Sciences and Techniques,\\
Moulay Ismail University of Meknes, Errachidia 52000, Morocco}



\begin{abstract}
This paper introduces a spatiotemporal SEIQR epidemic model governed by a system of reaction-diffusion partial differential equations that incorporates optimal control strategies. The model captures the transmission dynamics of an infectious disease across space and time. It includes three time-dependent control variables: preventive measures for susceptible individuals, quarantine for infectious individuals, and treatment for quarantined individuals. The study has four main objectives: (i) to prove the existence, uniqueness, and positivity of global strong solutions using analytic semigroup theory, (ii) to demonstrate the existence of optimal control strategies through functional analysis techniques, (iii) to derive first-order necessary optimality conditions via convex perturbation methods and adjoint equations, and (iv) to perform numerical simulations to assess the effectiveness of different combinations of control interventions. The simulation results emphasize the advantages of combining pharmaceutical and non-pharmaceutical interventions to minimize disease prevalence and control-related expenses.
\end{abstract}


\journal{} 




\end{frontmatter}


\section{Introduction}\label{S1}
Mathematical modeling of infectious diseases has played a vital role in understanding epidemic dynamics and informing public health strategies 
\cite{Bandekar2022, MohammedAwel2020, Kim2006, Ahmad2022, Zinihi2025MM}.
Classical compartmental models, such as the SIR and SEIR frameworks \cite{Korobeinikov2008, Kuznetsov1994, KhudaBukhsh2024, zinihi2025OC}, 
first introduced in the early 20th century by Kermack and McKendrick \cite{Kermack1927}, 
have provided foundational insights into the mechanisms underlying disease spread.
Over time, these models have evolved to include more refined structures that capture the complexity of real-world epidemics, such as latency periods, temporary immunity, asymptomatic transmission, and spatial heterogeneity \cite{Balderrama2022, Qiu2024, Zinihi2025FDE}.

Public health emergencies such as the ongoing SARS-CoV-2 pandemic, Ebola outbreaks, and the reemergence of diseases like measles and tuberculosis have recently sparked renewed interest in creating more realistic, spatially explicit epidemic models, cf.\
\cite{Bandekar2022, MohammedAwel2020, Zinihi2025MM, Korobeinikov2008, KhudaBukhsh2024, zinihi2025OC, Balderrama2022, Qiu2024, Zinihi2025FDE}.
Reaction-diffusion systems have proven to be powerful tools for capturing the temporal progression and spatial dissemination of infectious diseases.
By incorporating diffusion terms into compartmental models, one can account for individual movement and nonuniform infection risk distribution, both of which are essential for regional containment strategies.
The work \cite{Auricchio2023} addressed the well-posedness of a spatiotemporal system arising from an SEIR epidemiological model. 
The system consisted of four nonlinear partial differential equations (PDEs) with diffusion coefficients that depended on the total SEIR population.
The model was designed to capture the spatiotemporal dynamics of the SARS-CoV-2 pandemic.

In \cite{SidiAmmi2022}, the authors investigated a fractional reaction-diffusion version of the classical SIR epidemic model using the non-local, non-singular ABC fractional derivative.
The results highlighted the advantages of using the ABC fractional derivative compared to the classical integer-order formulation and emphasized the importance of choosing the appropriate fractional order to more accurately capture the system’s dynamics.
Wu and Zhao \cite{Wu2018} developed a nonlocal reaction-diffusion model with periodic delays to study the transmission dynamics of vector-borne diseases.
The model incorporated key biological and environmental factors, including spatial heterogeneity, seasonal variation, and the temperature dependence of the extrinsic and intrinsic incubation periods. 
The authors defined a basic reproduction number, $\mathcal{R}_0$, and established a threshold-type result characterizing the global dynamics of the system in terms of $\mathcal{R}_0$.
Another study, conducted by Zinihi et al.\ \cite{Zinihi2025FDE}, analyzed a fractional parabolic SIR epidemic model that incorporated the nonlocal Caputo derivative and the nonlinear spatial $p$-Laplacian operator. 
Immunity was introduced through a vaccination program treated as a control variable. 
The objective was to determine an optimal control pair that minimizes the number of infected individuals and the associated costs of vaccination and treatment within a bounded spatiotemporal domain.
Zinihi et al.\ \cite{Zinihi2025GS} investigated the global stability of a class of nonlinear parabolic equations, focusing particularly on applications in biology, especially epidemiology. Their analysis was based on constructing Lyapunov functions from the corresponding ordinary differential equations (ODEs).
To illustrate the methodology, the authors provided a representative example from epidemiology.

Of the many extensions of classical models, the SEIQR structure has emerged as a particularly relevant one in modern epidemiological studies.
The model divides the population into five categories: susceptible ($S$), exposed ($E$), infectious ($I$), quarantined ($Q$), and recovered ($R$).
This classification system is well-suited for capturing the delayed onset of infectiousness, targeted quarantine interventions, 
and the gradual return of individuals to the susceptible or recovered classes.
For example, Bhadauria, Devi and Gupta \cite{Bhadauria2021} proposed and analyzed an SEIQR epidemic model 
that incorporates a delay to account for the time between exposure and the onset of infection.
They examined the model using Lyapunov stability theory to assess local and global stability and performed 
Hopf bifurcation analysis to explore the emergence of periodic solutions.
The study highlighted the significant role of asymptomatic cases arising from the exposed population in accelerating the spread of SARS-CoV-2.
Srivastava and Nilam \cite{Srivastava2024} proposed a fractional SEIQR epidemic model that incorporates the Monod-Haldane incidence rate
to capture psychological effects and a saturated quarantine response modeled by a Holling type III functional form.
The system dynamics were governed by Caputo fractional derivatives to account for memory effects in disease transmission.
Separately, Huang \cite{Huang2016} introduced a novel optimization algorithm inspired by the SEIQR model, termed the SEIQRA (SEIQR Algorithm).
This bio-inspired approach simulates the spread of an infectious disease (e.g., SARS) within a population by mapping disease states onto computational agents, each of which is characterized by certain features.
Disease transmission and progression were analogized to state transitions controlled by thirteen specialized operators that performed actions such as averaging, reflection, and crossover.
Although several deterministic SEIQR models have been studied in recent literature, few have addressed the spatial dynamics of these models, especially when public health control measures are in place.

In practice, controlling diseases relies heavily on strategically deploying limited resources through prevention (e.g., vaccination or public awareness campaigns), quarantine, and treatment.
For example, Duan et al.\ \cite{Duan2020} proposed an SIRVS epidemic model that incorporates age-structured vaccination and recovery and accounts for vaccine- and infection-induced immunity. 
The study aimed to minimize costs associated with vaccination and treatment strategies by applying Pontryagin’s maximum principle when $R_0 > 1$. 
Zinihi, Sidi Ammi and Ehrhardt \cite{zinihi2025OC} investigated a fractional SEIR model that incorporates the non-singular Caputo–Fabrizio derivative. 
In this framework, vaccination was introduced as a control variable representing a public health intervention.
The primary objective of the study was to identify an optimal control pair that minimizes the number of infectious individuals and the total costs of vaccination and treatment. 
Another study by Liu, Xiang and Zhou \cite{Liu2024} addressed the dynamic behavior and optimal control of a delayed epidemic model. 
They formulated an optimal vaccination strategy to minimize the number of infected individuals, maximize the number of uninfected individuals, and reduce overall control costs. 
Similarly, Salwahan, Abbas and Tridane \cite{Salwahan2025} examined a modified SITR model governed by a periodically switched system, 
in which the transmission and treatment rates vary periodically. 
Due to the discontinuous dynamics induced by switching, the authors proposed a modified version of Pontryagin’s maximum principle to solve the associated optimal control problem. 
This study aimed to optimize vaccination strategies in dynamic, resource-constrained settings, validating the approach through numerical simulations.
Applying optimal control theory to epidemic models provides a rigorous framework for designing strategies that depend on time and space and minimize the health burden and economic cost of disease management. 
Combined with reaction-diffusion models, this approach enables comprehensive spatiotemporal optimization of control efforts.

Incorporating spatial heterogeneity into epidemic models is essential to capturing the real-world dynamics of disease spread, particularly when population mobility, local outbreaks, and regional disparities are significant factors. 
Using reaction-diffusion terms enables the model to represent spatial propagation driven by individual movement across geographic areas, which is particularly relevant in urban and semi-urban contexts.
Furthermore, managing an epidemic is complex and requires multiple control strategies because no single intervention is sufficient to effectively curb disease transmission. 
In this work, we consider pharmaceutical interventions, such as vaccination and quarantine, as well as non-pharmaceutical interventions, such as social distancing and awareness campaigns, to formulate a more realistic and adaptable control framework. 
Combining these strategies enables a flexible response: pharmaceutical interventions directly reduce the susceptible and infectious populations, while the non-pharmaceutical intervention modifies disease transmission pathways by reducing contact rates. 
This multifaceted control approach reflects the practical challenges health authorities face during epidemics, where decisions must balance limited medical resources with behavioral interventions that target public compliance and awareness.

In this work, we assume that population movement depends on time $t$ and space $x$. 
Furthermore, we denote the densities of the six compartments, which evolve with respect to time and space, by $S(t,x)$, $E(t,x)$, $I(t,x)$, $Q(t,x)$, and $R(t,x)$.
We propose a spatiotemporal SEIQR model formulated as a parabolic system of PDEs that incorporates the Laplacian operator to capture spatial movement. 
The model includes three control variables, each representing a specific public health intervention: prevention among susceptible individuals, quarantine of infected individuals, and treatment of quarantined individuals.

Our study has four main objectives. 
Section~\ref{S2} introduces the spatiotemporal SEIQR model, which is governed by a system of reaction-diffusion equations and incorporates three control variables corresponding to vaccination, quarantine, and treatment. 
Section~\ref{S3} establishes the existence, uniqueness, boundedness, and positivity of the model's solution using semigroup theory and functional analysis. 
Section~\ref{S4} formulates the optimal control problem, proves the existence of optimal strategies, and Section~\ref{S5} derives the first-order necessary conditions for optimality via an adjoint-based variational method. 
Section~\ref{S6} presents numerical simulations that illustrate the effectiveness of the proposed interventions and support the theoretical findings.
Finally, in Section~\ref{S7} we conclude and state perspectives for future research directions.

\section{Mathematical Model Formulation}\label{S2}
We propose a five-dimensional epidemiological model to describe the transmission dynamics of an epidemic.  
In this model, individuals in the population transition through five compartments over time: Susceptible ($S$), Exposed ($E$), Infected ($I$), Quarantined ($Q$), and Recovered ($R$). 
The transmission coefficients used in the SEIQR model are summarized in Table~\ref{Tab1}.

\begin{table}[H]
\centering
\setlength{\tabcolsep}{0.8cm}
\caption{Transmission coefficients for the proposed SEIQR reaction-diffusion model.}\label{Tab1}
\adjustbox{max width=\textwidth}{
\begin{tabular}{cc}
\hline 
\textbf{Symbol} & \textbf{Description} \\
\hline \hline 
$\Lambda$ & Recruitment rate (e.g.\ birth or immigration) \\
\hline
\multirow{2}{*}{$\beta_1$} & Transmission rate due to contact between\\ & susceptible and exposed individuals \\
\hline
\multirow{2}{*}{$\beta_2$} & Transmission rate due to contact between\\ & susceptible and infectious individuals \\
\hline
$\mu$ & Natural death rate \\
\hline
$\delta$ & Rate at which exposed individuals become infectious \\
\hline
$\gamma$ & Rate at which infectious individuals are quarantined \\
\hline
$\alpha$ & Recovery rate of quarantined individuals \\
\hline
\multirow{2}{*}{$\rho$} & Rate at which non-infected quarantined individuals\\ & return to the susceptible class \\
\hline
$\lambda_S$, $\lambda_E$, $\lambda_I$, $\lambda_Q$, $\lambda_R$ & Diffusion coefficients for $S$, $E$, $I$, $Q$, and $R$ respectively \\
\hline
\end{tabular}
}
\end{table}

The proposed SEIQR reaction-diffusion model describes the spatiotemporal spread of an infectious disease by modeling the movement and interactions of individuals across five compartments: susceptible ($S$), exposed ($E$), infected ($I$), quarantined ($Q$), and recovered ($R$).\\
\textbf{\textit{Susceptible}} ($S$): Susceptible individuals are at risk of contracting the disease. Their population increases through recruitment $\Lambda$ and re-entry from quarantine at a rate $\rho Q$, representing individuals who tested negative or were misclassified. 
They decrease due to natural mortality ($\mu$) and infection upon contact with exposed ($E$) and infected ($I$) individuals, at rates $\beta_1 S E$ and $\beta_2 S I$, respectively. 
Spatial movement is modeled by the diffusion term $\lambda_S \Delta S$.\\
\textbf{\textit{Exposed}} ($E$): Exposed individuals have been infected but are not yet infectious. This group increases via contact between susceptibles and exposed individuals, governed by the transmission term $\beta_1 S E$. 
Exposed individuals either progress to the infected class at rate $\delta$ or die naturally. 
Spatial diffusion is represented by $\lambda_E \Delta E$.\\
\textbf{\textit{Infected}} ($I$): Infected individuals are symptomatic and capable of transmitting the disease. Their number increases through contact with susceptible individuals ($\beta_2 S I$) and through progression from the exposed class ($\delta E$). 
They may be quarantined at rate $\gamma$ or die naturally at rate $\mu$. Their spatial spread is modeled by $\lambda_I \Delta I$.\\
\textbf{\textit{Quarantined}} ($Q$): This compartment consists of individuals who are isolated after developing symptoms. They transition from the infected class at a rate $\gamma$ and may recover ($\alpha Q$), die naturally ($\mu Q$), or return to the susceptible class ($\rho Q$) if found uninfected. 
Their spatial redistribution is governed by $\lambda_Q \Delta Q$.\\
\textbf{\textit{Recovered}} ($R$): Recovered individuals have acquired immunity after completing quarantine. They accumulate at a rate $\alpha Q$ and are subject to natural death at a rate $\mu$.
Their spatial mobility is described by $\lambda_R \Delta R$.

Throughout the model,  the diffusion terms $\lambda_j \Delta X$ for each compartment $X$ reflect spatial spread due to individual movement. 
The mortality terms $\mu X$ account for natural deaths unrelated to the disease.
By incorporating local disease dynamics and spatial processes, 
the model captures key epidemic propagation mechanisms, including localized outbreaks, spatial heterogeneity, and the potential impact of quarantine interventions.
Figure~\ref{F1} illustrates the compartmental structure of the SEIQR model and highlights key transitions between health states, including infection, quarantine, and recovery.

\begin{figure}[H]
\centering
\begin{tikzpicture}[node distance=4cm]
\node (S) [rectangle, draw, minimum size=1cm, fill=cyan!30] {S};
\node (I) [rectangle, draw, minimum size=1cm, fill=red!30, right of=S] {I};
\node (Q) [rectangle, draw, minimum size=1cm, fill=blue!30, right of=I] {Q};
\node (E) [rectangle, draw, minimum size=1cm, fill=orange!30, xshift = 4cm, yshift = 2.5cm] {E};
\node (R) [rectangle, draw, minimum size=1cm, fill=green!30, right of=Q] {R};
\draw[->] (-1.5,0) -- ++(S) node[midway,above]{$\Lambda$};
\draw [->] (S.north) -- (0,2.5) -- (E) node[midway,below]{$\beta_1 S E$};
\draw [->] (S) -- (I) node[midway,below]{$\beta_2 S I$};
\draw [->] (4.1,2) -- (4.1,0.5) node[midway,right]{$\delta E$};
\draw [->] (I) -- (Q) node[midway,below]{$\gamma I$};
\draw [->] (Q) -- (R) node[midway,below]{$\alpha Q$};
\draw [->] (7.9,-0.5) -- (7.9,-1.7) -- (0.1,-1.7) node[midway,below]{$\rho Q$} -- (0.1,-0.5);
\draw[->] (-0.1,-0.5) -| (-0.1,-1.5) node[near end,left]{$\mu S$};
\draw[->] (4.5,2.5) -- (5.7,2.5) node[midway,below]{$\mu E$};
\draw[->] (3.9,0.5) -| (3.9,1.5) node[near end,left]{$\mu I$};
\draw[->] (8.1,-0.5) -| (8.1,-1.5) node[near end,right]{$\mu Q$};
\draw[->] (R.south) -| (12,-1.5) node[near end,right]{$\mu R$};
\end{tikzpicture}
\captionof{figure}{Transmission pathways in the proposed SEIQR model.}\label{F1}
\end{figure}
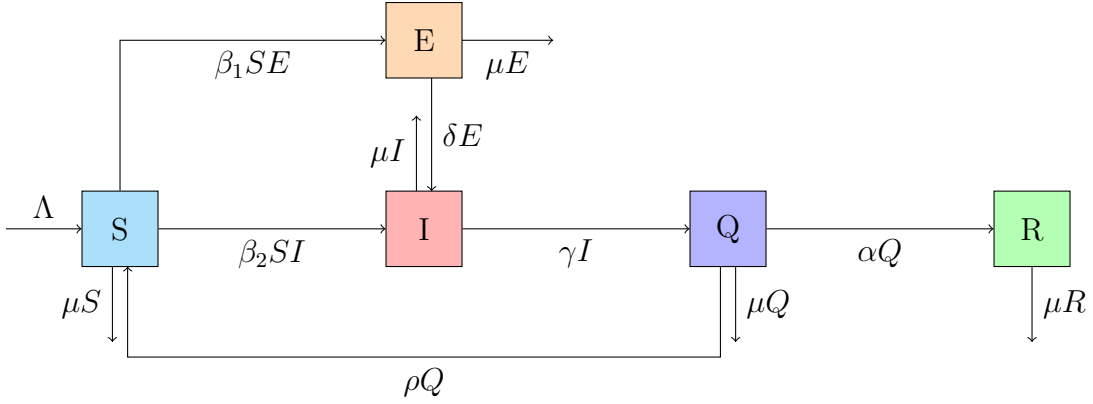

Let $\Omega \subset \mathbb{R}^n$ be a bounded domain with a smooth boundary $\partial \Omega$, 
where $n = 2$ or $n = 3$. 
The spatiotemporal dynamics of the proposed SEIQR model are described mathematically by the following system of reaction-diffusion equations
\begin{equation*}
\left\{\begin{aligned}
    \frac{\partial S(t, x)}{\partial t} &= \lambda_S \Delta S(t, x) + \Lambda + \rho Q(t, x) - \beta_1 S(t, x) E(t, x) - \beta_2 S(t, x) I(t, x) - \mu S(t, x), \\
    \frac{\partial E(t, x)}{\partial t} &= \lambda_E \Delta E(t, x) + \beta_1 S(t, x) E(t, x) - (\delta + \mu) E(t, x), \\
     \frac{\partial I(t, x)}{\partial t} &= \lambda_I \Delta I(t, x) + \beta_2 S(t, x) I(t, x) + \delta E(t, x) - (\gamma + \mu) I(t, x), \\
      \frac{\partial Q(t, x)}{\partial t} &= \lambda_Q \Delta Q(t, x) + \gamma I(t, x) - (\alpha + \rho + \mu) Q(t, x), \\
     \frac{\partial R(t, x)}{\partial t} &= \lambda_R \Delta R(t, x) +  \alpha Q(t, x) - \mu R(t, x),
\end{aligned}\right. \text{ in } \mathcal{U}, 
\end{equation*}
with
\begin{equation*}
\left\{\begin{aligned}
&\nabla S\cdot \vec{n} = \nabla E\cdot \vec{n} = \nabla I\cdot \vec{n} = \nabla Q\cdot \vec{n} = \nabla R\cdot \vec{n} = 0,  & &\text{ on } \Sigma_T,\\
&S(0,x)=S_0, \ E(0,x)=E_0, \ I(0,x)=I_0, \ Q(0,x)=Q_0, \ R(0,x)=R_0,  & &\text{ in } \Omega,
\end{aligned}\right.
\end{equation*}
where $T>0$, $\mathcal{U} = [0, T]\times\Omega$, $\vec{n}$ is the normal vector to the boundary $\Sigma_T = [0, T]\times\partial\Omega$. 
Imposing homogeneous Neumann (no-flux) boundary conditions ensures that the SEIQR mode is self-contained, with dynamics driven entirely by internal processes, and no movement occurs across the boundary $\partial \Omega$. Additionally, the initial data for all compartments are positive throughout the domain $\Omega$.

\subsection{Controlled SEIQR Reaction-Diffusion Model}
We now formulate an optimal control model to manage the spread of infectious diseases. 
This model uses a reaction-diffusion framework that accounts for spatial dynamics and disease evolution within a population.
We introduce three control strategies: 
1) \textit{Pharmaceutical Interventions} (PI), represented by controls $u_1$ and $u_2$; 
and (2) \textit{Non-Pharmaceutical Interventions} (NPI), represented by control $u_3$. 
Specifically, $u_1$ models vaccination efforts that move susceptible individuals directly to the recovered class; $u_2$ reflects treatment strategies that accelerate the recovery of quarantined individuals; and $u_3$ represents social distancing measures that reduce transmission between susceptible, exposed, and infectious individuals. 
The goal is to minimize the overall cost of disease spread while ensuring the most efficient allocation of resources to control the epidemic.
The model captures the spatial spread of the disease via diffusion terms and incorporates control efforts to reduce transmission, speed recovery, and manage quarantine effectively.
Figure~\ref{F2}illustrates the updated compartmental structure of the optimal control SEIQR model based on Table~\ref{Tab1}.

\begin{figure}[H]
\centering
\begin{tikzpicture}[node distance=4cm]
\node (S) [rectangle, draw, minimum size=1cm, fill=cyan!30] {S};
\node (I) [rectangle, draw, minimum size=1cm, fill=red!30, right of=S] {I};
\node (Q) [rectangle, draw, minimum size=1cm, fill=blue!30, right of=I] {Q};
\node (E) [rectangle, draw, minimum size=1cm, fill=orange!30, xshift = 4cm, yshift = 2.5cm] {E};
\node (R) [rectangle, draw, minimum size=1cm, fill=green!30, right of=Q] {R};
\draw[->] (-1.5,0.1) -- (-0.5,0.1) node[midway,above]{$\Lambda$};
\draw [->] (S.north) -- (0,2.5) -- (E) node[midway,above]{$\beta_1 (1 - u_3) S E$};
\draw [->] (S) -- (I) node[midway,above]{$\beta_2 (1 - u_3) S I$};
\draw [->] (4.1,2) -- (4.1,0.5) node[midway,right]{$\delta E$};
\draw [->] (I) -- (Q) node[midway,above]{$\gamma I$};
\draw [->] (8.5,0.1) -- (11.5,0.1) node[midway,above]{$\alpha Q$};
\draw [->] (8.5,-0.1) -- (11.5,-0.1) node[midway,below]{$u_2 Q$};
\draw [->] (7.9,-0.5) -- (7.9,-1.5) -- (0.1,-1.5) node[midway,above]{$\rho Q$} -- (0.1,-0.5);
\draw [->] (-0.1,-0.5) -- (-0.1,-2.2) -- (11.9,-2.2) node[midway,above]{$u_1 S$} -- (11.9,-0.5);
\draw[->] (-0.5,-0.1) -- (-1.5,-0.1) node[midway,below]{$\mu S$};
\draw[->] (4.5,2.5) -- (5.7,2.5) node[midway,above]{$\mu E$};
\draw[->] (3.9,0.5) -| (3.9,1.5) node[near end,left]{$\mu I$};
\draw[->] (8.1,-0.5) -| (8.1,-1.5) node[near end,right]{$\mu Q$};
\draw[->] (12.1,-0.5) -| (12.1,-1.5) node[near end,right]{$\mu R$};
\end{tikzpicture}
\captionof{figure}{Transmission pathways in the updated SEIQR model.}\label{F2}
\end{figure}
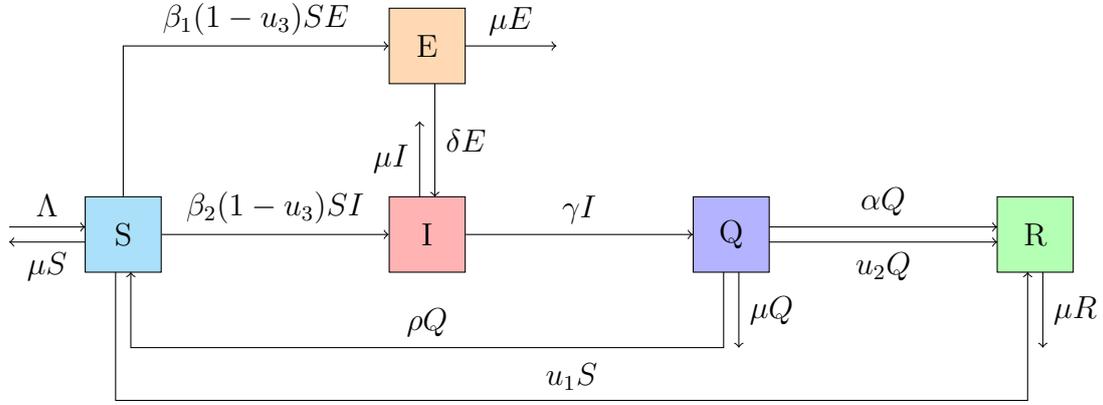

We aim to identify strategies that balance public health outcomes with resource constraints by optimizing these controls over time and space. 
This will provide valuable insights into epidemic management in real-world settings. 
The optimal control SEIQR model is mathematically formulated as follows:
\begin{equation}\label{E2.1}
\left\{\begin{aligned}
\frac{\partial S(t, x)}{\partial t} &= \lambda_S \Delta S(t, x) + \Lambda + \rho Q(t, x) - (\mu + u_1(t, x)) S(t, x)  \\
& \qquad- \beta_1 (1 - u_3(t, x)) S(t, x) E(t, x) - \beta_2 (1 - u_3(t, x)) S(t, x) I(t, x), \\
\frac{\partial E(t, x)}{\partial t} &= \lambda_E \Delta E(t, x) + \beta_1 (1 - u_3(t, x)) S(t, x) E(t, x) - (\delta + \mu) E(t, x), \\
\frac{\partial I(t, x)}{\partial t} &= \lambda_I \Delta I(t, x) + \beta_2 (1 - u_3(t, x)) S(t, x) I(t, x) + \delta E(t, x) - (\gamma + \mu) I(t, x), \\
\frac{\partial Q(t, x)}{\partial t} &= \lambda_Q \Delta Q(t, x) + \gamma I(t, x) - (\alpha + \rho + \mu + u_2(t, x)) Q(t, x), \\
\frac{\partial R(t, x)}{\partial t} &= \lambda_R \Delta R(t, x) +  \alpha Q(t, x) + u_1(t, x) S(t, x) + u_2(t, x) Q(t, x)- \mu R(t, x),
\end{aligned}\right. \text{ in } \mathcal{U}, 
\end{equation}
with
\begin{equation}\label{E2.2}
\left\{\begin{aligned}
   &\nabla S\cdot \vec{n} = \nabla E\cdot \vec{n} = \nabla I\cdot \vec{n} = \nabla Q\cdot \vec{n} = \nabla R\cdot \vec{n} = 0,  & &\text{ on } \Sigma_T,\\
   &S(0,x)=S_0, \ E(0,x)=E_0, \ I(0,x)=I_0, \ Q(0,x)=Q_0, \ R(0,x)=R_0,  & &\text{ in } \Omega.
\end{aligned}\right.
\end{equation}
Each equation describes the rate at which a compartment changes over time and space. 
These equations capture the interactions among compartments based on contact rates, transition dynamics, and natural processes, such as disease progression and recovery.

Table~\ref{Tab2} summarizes the role and impact of each control variable, detailing the type of intervention, the targeted compartments, and the intended epidemiological effects. These controls are designed to be adjusted over time $t$ and space $x$ to reflect practical implementation strategies and optimize epidemic outcomes.

\begin{table}[H]
\centering
\setlength{\tabcolsep}{0.32cm}
\caption{Description of control variables used in the optimal control model.}\label{Tab2}
\adjustbox{max width=\textwidth}{
\begin{tabular}{c||ccc}
\hline
\textbf{Control} & $ u_1(t,x) $ & $ u_2(t,x) $ & $ u_3(t,x) $\\
\hline\hline
\textbf{Type} & PI & PI & NPI\\
\hline
\multirow{2}{*}{\textbf{Description}} & \multirow{2}{*}{Vaccination} & \multirow{2}{*}{Treatment} & Social Distancing\\
& & & Public Awareness\\
\hline
\textbf{Target} & \multirow{2}{*}{$S$, $R$} & \multirow{2}{*}{$Q$, $R$} & \multirow{2}{*}{$S$, $E$, $I$}\\
\textbf{Compartment(s)} & & &\\
\hline
\multirow{3}{*}{\textbf{Effect}} & \multirow{2}{*}{Increases immunity} & Accelerates recovery & Reduces effective\\
& \multirow{2}{*}{in the population} & for quarantined & contact rates by\\
& & individuals & lowering interactions\\
\hline
\end{tabular}
}
\end{table}

\subsection{Quantification of Epidemic and Control Costs}
Our primary objective is to minimize the number of individuals who are exposed, infected, or quarantined, while reducing the overall cost of vaccination and treatment within the time interval $[0, T]$.
In this subsection, we evaluate the total costs of implementing spatiotemporal control measures to mitigate the spread and impact of the epidemic while minimizing related expenses in time and space.

\textbf{\textit{Cost due to disease impact:}} 
This component captures the economic consequences of disease transmission, These consequences include productivity losses and healthcare burdens associated with exposed and infected individuals. It is represented by
\begin{equation*}
  \int_0^T \int_\Omega \kappa_1 (E(t,x) + I(t,x)) \, dx \, dt + \int_\Omega \kappa_3 (E(T,x) + I(T,x)) \, dx,
\end{equation*}
where $\kappa_1$ and $\kappa_3$ denote the respective weights associated with the prevalence of the disease during the control horizon and at the final time.

\textbf{\textit{Cost of quarantine and management:}} 
This cost accounts for the economic and logistical expenses related to isolating infected individuals, including quarantine facilities, monitoring, and care. 
It is given by
\begin{equation*}
   \int_0^T \int_\Omega \kappa_2 Q(t,x) \, dx \, dt + \int_\Omega \kappa_4 Q(T,x) \, dx,
\end{equation*}
where $\kappa_2$ and $\kappa_4$ measure the burden of quarantine over time and at the terminal state.

\textbf{\textit{Cost of vaccination campaigns:}} 
The control $u_1(t,x)$ represents vaccination efforts, with associated costs stemming from vaccine production, distribution, and administration. 
This is expressed as
\begin{equation*}
   \int_0^T \int_\Omega w_1 u_1(t,x) \, dx \, dt + \int_\Omega \sigma_1 u_1(T,x) \, dx,
\end{equation*}
where $w_1$ and $\sigma_1$ are weights reflecting operational and terminal costs of vaccination.

\textbf{\textit{Cost of quarantine, medical care, and isolation protocols:}} 
The control $u_2(t,x)$ reflects intensified isolation and treatment strategies for quarantined individuals, such as enhanced medical intervention and specialized care. 
Its cost is measured by $w_2$ during the control horizon and by $\sigma_2$ at the terminal state:
\begin{equation*}
\int_0^T \int_\Omega w_2 u_2(t,x) \, dx \, dt + \int_\Omega \sigma_2 u_2(T,x) \, dx.
\end{equation*}

\textbf{\textit{Cost of preventive measures and public health interventions:}}
The control $u_3(t,x)$ models efforts aimed at reducing transmission through behavioral interventions such as social distancing, mask mandates, and hygiene promotion. 
These actions require continuous investment in awareness campaigns and infrastructure, 
which incur a cost given by
\begin{equation*}
   \int_0^T \int_\Omega w_3 u_3(t,x) \, dx \, dt + \int_\Omega \sigma_3 u_3(T,x) \, dx,
\end{equation*}
where $w_3$ reflects the cost of sustained public health efforts and $\sigma_3$ accounts for the residual cost or effort required at the end of the intervention period.

Combining all these contributions, the total objective functional to be minimized is given by
\begin{equation}\label{E2.3}
  \mathcal{J}(S, E, I, Q, R, u_1, u_2, u_3) = \int_0^T \int_\Omega \mathcal{A}_1(t, x) \, dx \, dt + \int_\Omega \mathcal{A}_2(T, x) \, dx,
\end{equation}
   subjected to \eqref{E2.1}--\eqref{E2.2}, where
\begin{equation*}
\begin{aligned}
    & \bullet \mathcal{A}_1(t, x) = \kappa_1 (E + I)(t,x) + \kappa_2 Q(t,x) + w_1 u_1(t,x) + w_2 u_2(t,x) + w_3 u_3(t,x),\\
    & \bullet \mathcal{A}_2(T, x) = \kappa_3 (E + I)(T,x) + \kappa_4 Q(T,x) + \sigma_1 u_1(T,x) + \sigma_2 u_2(T,x) + \sigma_3 u_3(T,x).
\end{aligned}
\end{equation*}

Let $(S, E, I, Q, R)$ be the solution to the system~\eqref{E2.1}--\eqref{E2.2}.
The optimal control problem is to minimize the objective functional $\mathcal{J}$; that is, to find a control function  
\begin{equation*}
     u^* = (u_1^*, u_2^*, u_3^*) \in \mathcal{V}_{ad} = \bigl\{(u_1, u_2, u_3) \in (L^2(\mathcal{U}))^3 \ | \ 0 \leq u_i(t, x) \leq 1 \ \text{a.e.\ in } \mathcal{U}, \ \forall i = 1, 2, 3 \bigr\},
\end{equation*}
such that
\begin{equation*}
    \mathcal{J}(S^*, E^*, I^*, Q^*, R^*, u_1^*, u_2^*, u_3^*) 
    = \inf_{(u_1, u_2, u_3) \in \mathcal{V}_{ad}} \mathcal{J}(S, E, I, Q, R, u_1, u_2, u_3).
\end{equation*}

\section{Well-Posedness Analysis}\label{S3}
In this section, we use the $C_0$-semigroup theory to prove the boundedness, positivity, existence, and uniqueness of solutions to the proposed system~\eqref{E2.1}--\eqref{E2.2}.
To do so, let $\psi = (\psi_i)_{1 \leq i \leq 5} = (S, E, I, Q, R)$, 
$\psi^0 =( \psi^0_i)_{1 \leq i \leq 5}$, 
$\lambda = (\lambda_i)_{1 \leq i \leq 5} = (\lambda_S, \lambda_E, \lambda_I, \lambda_Q, \lambda_R)$, 
$\mathbb{E}(\Omega) = \bigl(L^2(\Omega)\bigr)^5$ denotes the Hilbert space of 5-tuples of square-integrable functions on $\Omega$, corresponding to the components $(S, E, I, Q, R)$, and $\mathcal{T}$ defined as follows
\begin{equation*}
\begin{aligned}
    \mathcal{T}\colon \ \mathcal{D}_\mathcal{T} 
    = \Big\{\vartheta \in \bigl( H^2(\Omega) \bigr)^5 \mid \ \nabla \vartheta_i\cdot \vec{n} = 0,\; 1\leq i\leq 5\Big\} \subset \mathbb{E}(\Omega) &\longrightarrow \mathbb{E}(\Omega),\\ 
    \vartheta &\longrightarrow -\lambda\Delta \vartheta = (-\lambda_i\Delta \vartheta_i)_{i=1,2,3,4,5}.
\end{aligned}
\end{equation*}
We introduce the function $\mathcal{F}$, which is defined by
\begin{equation*}
    \mathcal{F}\bigl(\psi(t)\bigr)=\bigl(\mathcal{F}_1(\psi(t)),\,\mathcal{F}_2(\psi(t)),\,\mathcal{F}_3(\psi(t)),\,\mathcal{F}_4(\psi(t)),\,\mathcal{F}_5(\psi(t))\bigr), \ t\in [0, T],
\end{equation*}
with $\psi(t)(\cdot) = \psi(t,\cdot)$, and 
\begin{equation*}
   \left\{\begin{aligned}
   \mathcal{F}_1(\psi(t)) &= \Lambda + \rho \psi_4 - \beta_1 (1-u_3) \psi_1\psi_2 - \beta_2 (1-u_3) \psi_1\psi_3 - (\mu+u_1) \psi_1,\\
   \mathcal{F}_2(\psi(t)) &= \beta_1 (1-u_3) \psi_1\psi_2 - (\mu + \delta)\psi_2,\\
    \mathcal{F}_3(\psi(t)) &= \beta_2 (1-u_3) \psi_1\psi_3 + \delta\psi_2 - (\mu + \gamma)\psi_2,\\
     \mathcal{F}_4(\psi(t)) &= \gamma \psi_3 - (\mu + \alpha + \rho + u_2)\psi_4,\\
     \mathcal{F}_5(\psi(t)) &= \alpha \psi_4 + u_1\psi_1 + u_2 \psi_4 - \mu \psi_5.
\end{aligned}\right.
\end{equation*}
Subsequently, the system~\eqref{E2.1}--\eqref{E2.2} can be reformulated in $\mathbb{E}(\Omega)$ as follows
\begin{equation}\label{E3.1}
  \left\{\begin{aligned}
   &\partial_t \psi(t) + \mathcal{T}\psi(t) = \mathcal{F}(\psi(t)),\\
   &\psi(0)=\psi^0.
\end{aligned}\right. \quad t\in[0, T],
\end{equation}
First, we will prove the boundedness of the function $\psi$, as stated in the following result.

\begin{proposition}\label{P1}
All solutions of the SEIQR model~\eqref{E3.1} are bounded. 
\end{proposition}
\begin{proof}
Let $i \in \{1, 2, 3, 4, 5\}$, $\{\mathcal{S}_i(t), t \geq 0\}$ be the $C_0$-semigroup generated by the operator component $\mathcal{T}_i$, 
where $\mathcal{T}_i \psi_i = \lambda_i\Delta \psi_i$. Then, we set
\begin{equation*}
    \eta_i = \max\Bigl\{ \bigl\|\mathcal{F}_i\bigr\|_{L^{\infty}(\mathcal{U}))}, \bigl\|\psi_i^0\bigr\|_{L^{\infty}(\Omega)} \Bigr\}.
\end{equation*}
It is clear that the function
\begin{equation*}
    \varphi_i = \psi_i - \eta_i t - \bigl\|\psi_i^0\bigr\|_{L^{\infty}(\Omega)},
\end{equation*}
satisfies
\begin{equation}\label{E3.2}
   \left\{\begin{aligned}
    &\partial_t \varphi_i = \lambda_i \Delta \varphi_i + \mathcal{F}_i(\psi(t)) - \eta_i,  \\
    &\varphi_i(0, x) = \psi_i^0 -\bigl\|\psi_i^0\bigr\|_{L^{\infty}(\Omega)}.
\end{aligned}\right. \quad t \in[0, T],
\end{equation}
The strong solution of \eqref{E3.2} takes the form
\begin{equation*}
   \varphi_i(t) = \mathcal{S}_i(t) \Bigl(\psi_i^0 - \bigl\|\psi_i^0\bigr\|_{L^{\infty}(\Omega)} \Bigr) 
    +\int_0^t \mathcal{S}_i(t-s) \bigl(\mathcal{F}_i(\psi(t))- \eta_i\bigr) \,ds.
\end{equation*}
Since $\mathcal{F}_i(\psi(t))- \eta_i \leq 0$ and $\psi_i^0-\bigl\|\psi_i^0\bigr\|_{L^{\infty}(\Omega)} \leq 0$, it follows that $\varphi_i \leq 0$ for all $(t, x) \in \mathcal{U}$.\\
Similarly, the function 
\begin{equation*}
    \Tilde{\varphi}_i = \psi_i + \eta_i t + \bigl\|\psi_i^0\bigr\|_{L^{\infty}(\Omega)},
\end{equation*}
satisfies
\begin{equation*}
    \left\{\begin{aligned}
     &\partial_t \Tilde{\varphi}_i (t, x)= \lambda_i \Delta \Tilde{\varphi}_i + \mathcal{F}_i(\psi(t))+ \eta_i,\\
     &\Tilde{\varphi}_i(0, x)=\psi_i^0 + \bigl\|\psi_i^0\bigr\|_{L^{\infty}(\Omega)}.
\end{aligned}\right. \quad t \in[0, T].
\end{equation*}
Thus, 
\begin{equation*}
\Tilde{\varphi}_i(t) = \mathcal{S}_i(t) \Bigl(\psi_i^0 + \bigl\|\psi_i^0\bigr\|_{L^{\infty}(\Omega)} \Bigr) 
    + \int_0^t \mathcal{S}_i(t-s) \bigl(\mathcal{F}_i(\psi(t)) + \eta_i\bigr) \,ds.
\end{equation*}
Since $\mathcal{F}_i(\psi(t)) + \eta_i \geq 0$ and $\psi_i^0 + \bigl\|\psi_i^0\bigr\|_{L^{\infty}(\Omega)} \geq 0$, it follows that $\Tilde{\varphi}_i \geq 0$ for all $(t, x) \in \mathcal{U}$.\\
Consequently,
\begin{equation*}
    \bigl|\psi_i(t, x)\bigr| \leq \eta_i t + \bigl\|\psi_i^0\bigr\|_{L^{\infty}(\Omega)}.
\end{equation*}
Therefore, we have $\psi \in \bigl(L^{\infty}(\mathcal{U})\bigr)^5$.
\end{proof}

Now, let us prove that $\psi$ is nonnegative using Cauchy-Schwarz and Gronwall's inequalities, 
as stated in the following proposition.
\begin{proposition}\label{P2}
For all positive initial data associated with the system~\eqref{E2.1}--\eqref{E2.2}, the solutions of the SEIQR model remain positive.
\end{proposition}
\begin{proof}
We start considering $\psi_4$. It can be observed that $\psi_4 = \psi_4^+ - \psi_4^-$, where we defined
\begin{equation*}
     \psi_4^+(t, x) = \sup\{\psi_4(t, x), 0\} \quad\text{and}\quad \psi_4^-(t, x) = \sup\{-\psi_4(t, x), 0\}.
\end{equation*}
Multiplying the equation corresponding to $i=4$ in~\eqref{E3.1} by $ \psi_4^- $, we obtain
\begin{equation*}
   -\frac{1}{2} \frac{d}{d t} \bigl\|\psi_4^-\bigr\|_{L^2(\Omega)}^2
   = \lambda_4\int_{\Omega}\bigl|\nabla \psi_4^-\bigr|^2 \,dx
   - \gamma\int_{\Omega} \psi_3 \psi_4^- \,dx
    + (\alpha + \rho + \mu + u_2)\int_{\Omega}  (\psi_4^-)^2 \,dx .
\end{equation*}
Then, we have
\begin{equation*}
    \frac{d}{d t} \bigl\|\psi_4^-\bigr\|_{L^2(\Omega)}^2 \leq 
        2\gamma\int_{\Omega} \psi_3 \psi_4^- \,dx .
\end{equation*}
Next, applying the Cauchy-Schwarz inequality, we get
\begin{equation*}
   \frac{d}{d t} \bigl\|\psi_4^-\bigr\|_{L^2(\Omega)}^2 
   \leq c \bigl\|\psi_3\bigr\|_{L^2(\Omega)}^2 \bigl\|\psi_4^-\bigr\|_{L^2(\Omega)}^2.
\end{equation*}
By Proposition~\ref{P1}, we have $\psi_i \in L^{\infty}(\mathcal{U})$ for all $i \in \{1, 2, 3, 4, 5\}$. Then, there exists another constant $c > 0$ such that
\begin{equation*}
   \frac{d}{dt} \bigl\|\psi_4^-\bigr\|_{L^2(\Omega)}^2 
    \leq c \bigl\|\psi_4^-\bigr\|_{L^2(\Omega)}^2.
\end{equation*}
Using Gronwall's inequality, we obtain
\begin{equation*}
      \bigl\|\psi_4^-\bigr\|_{L^2(\Omega)}^2 \leq 0,
\end{equation*}
which leads to the conclusion that $\psi_4^- = 0$. 
Consequently, $\psi_4 \geq 0$ for all $(t, x) \in \mathcal{U}$.

Using a similar methodology applied to $\psi_2$, we derive the following expression
\begin{equation*}
   -\frac{1}{2} \frac{d}{d t}\bigl\|\psi_2^-\bigr\|_{L^2(\Omega)}^2 =   
   \lambda_2\int_{\Omega}\bigl|\nabla \psi_2^-\bigr|^2 \,dx 
    - \beta_1 (1 - u_3) \int_{\Omega} \psi_1 (\psi_2^-)^2 \,dx 
    + (\delta + \mu)\int_{\Omega} (\psi_2^-)^2 \,dx,
\end{equation*}
which can be expressed as
\begin{equation*}
   \frac{d}{d t} \bigl\|\psi_2^-\bigr\|_{L^2(\Omega)}^2 
   \leq 2\beta_1 (1-u_3) \int_{\Omega} \psi_1 (\psi_2^-)^2 \,dx
    \leq 2\beta_1 (1-u_3) N \bigl\|\psi_2^-\bigr\|_{L^2(\Omega)}^2,
\end{equation*}
where $N \in L^{\infty}(\mathcal{U})$ is the total population. Since $(u_1, u_2, u_3) \in \mathcal{V}_{ad}$, then again by Gronwall's inequality, we obtain
\begin{equation*}
    \bigl\|\psi_2^-\bigr\|_{L^2(\Omega)}^2 \leq 0,
\end{equation*}
which implies $\psi_2^- = 0$, and thus $\psi_2 \geq 0$. 
Using the same technique, we conclude that $\psi_1, \psi_3, \psi_5 \geq 0$.
\end{proof}

Let $(t,x)\in \mathcal{U}$. 
The function $\mathcal{F}$ is Lipschitz continuous in $\psi$, 
uniformly with respect to $t\in[0,T]$.
Moreover, the operator $-\Delta$ is strongly elliptic, cf.\ \cite{Pazy1983, Zinihi2025CH}. 
According to standard results in functional analysis (see, for example \cite{Barbu1994, Pazy1983, Zinihi2025CH}), the system~\eqref{E3.1} admits a unique strong solution $\psi \in W^{1,2}(0, T; \mathbb{E}(\Omega))$, satisfying the regularity conditions
\begin{equation*}
   \psi_i \in L^2(0, T; H^2(\Omega)) \cap L^\infty(0, T; H^1(\Omega)),\quad \forall i \in \{1, 2, 3, 4, 5\}.
\end{equation*}
This foundational result leads us to the following corollary.

\begin{corollary}\label{C1}
Problem~\eqref{E3.1} has a unique bounded positive global solution $\psi\in W^{1,2} (0, T; \mathbb{E}(\Omega))$. 
Furthermore,
\begin{equation*}
    \psi_i \in W^{1,2}\bigl(0, T; L^2(\Omega)\bigr)  \cap L^2\bigl(0, T; H^2(\Omega)\bigr) \cap L^\infty\bigl(0, T; H^1(\Omega)\bigr) \cap L^\infty(\mathcal{U})\quad \forall i\in\{1, 2, 3, 4, 5\}.
\end{equation*}
\end{corollary}

\begin{proposition}\label{P3}
Let $\psi$ be the solution of \eqref{E3.1}. Then
\begin{equation*}
  \Bigl\|\frac{\partial \psi_i}{\partial t}\Bigr\|_{L^2(\mathcal{U})}
   +\bigl\|\psi_i\bigr\|_{L^2(0, T; H^2(\Omega))}
   +\bigl\|\psi_i\bigr\|_{H^1(\Omega)} 
   +\bigl\|\psi_i\bigr\|_{L^{\infty}(\mathcal{U})} < \infty.
\end{equation*}
\end{proposition}

\begin{proof}
The first equation of \eqref{E3.1} gives
\begin{equation*}
\begin{aligned}
   \int_0^t \int_{\Omega} \Bigl|\frac{\partial \psi_1}{\partial  \tau}\Bigr|^2 \,d\tau dx 
   & -2 \lambda_1 \int_0^t \int_{\Omega} \frac{\partial \psi_1}{\partial \tau} \Delta \psi_1 \,d\tau dx + \lambda_1^2 \int_0^t \int_{\Omega} |\Delta \psi_1|^2 \,d\tau dx \\
   =& \int_0^t \int_{\Omega} \bigl(\Lambda + \rho \psi_4 - \beta_1 (1-u_3) \psi_1\psi_2 - \beta_2 (1-u_3) \psi_1\psi_3 - (\mu+u_1)\psi_1 \bigr)^2 \,d\tau dx.
\end{aligned} 
\end{equation*}
Because of
\begin{equation*}
  \int_0^t \int_{\Omega} \frac{\partial \psi_1}{\partial \tau} \Delta \psi_1 \,d\tau dx
  = \int_{\Omega}\Bigl(-\bigl|\nabla \psi_1\bigr|^2 + \bigl|\nabla \psi_1^0\bigr|^2\Bigr) \,dx,
\end{equation*}
we have
\begin{equation*}
\begin{aligned}
    \int_0^t \int_{\Omega} \Bigl|\frac{\partial \psi_1}{\partial \tau}\Bigr|^2 \,d\tau dx &+ \lambda_1^2 \int_0^t \int_{\Omega} \bigl|\Delta \psi_1\bigr|^2 \,d\tau dx + 2 \lambda_1 \int_{\Omega} \bigl|\nabla \psi_1\bigr|^2 \,dx
   -2 \lambda_1 \int_{\Omega} \bigl|\nabla \psi_1^0\bigr|^2 \,dx \\
   =& \int_0^t \int_{\Omega} \bigl(\Lambda + \rho \psi_4 - \beta_1 (1-u_3) \psi_1\psi_2 - \beta_2 (1-u_3) \psi_1\psi_3 - (\mu+u_1)\psi_1 \bigr)^2 \,d\tau dx.
\end{aligned}
\end{equation*}
Due to the boundedness of $\psi_1^0 \in H^2(\Omega)$ and $\|\psi_i\|_{L^{\infty}(Q)}$, the result holds for i = 1. 
Analogous reasoning can be applied to the remaining scenarios.
\end{proof}

\section{Existence of an Optimal Solution}\label{S4}
This section proves the existence of an optimal solution pair, $(\psi^*, u^*)$ for the optimal control problem, Eq.~\eqref{E3.1}, using the technique of minimizing sequences, where $u^* = (u_1^*, u_2^*, u_3^*)$. The main result of this section is formulated and proved as follows:
\begin{theorem}
Problem~\eqref{E2.1}--\eqref{E2.2} admits at least an optimal solution pair $(\psi^*,u^*)$ such that  
\begin{equation*}
    \mathcal{J}(\psi^*, u^*) = \inf_{u \in \mathcal{V}_{\text{ad}}} \mathcal{J}(\psi, u).
\end{equation*}
\end{theorem}
\begin{proof}
From Corollary~\ref{C1}, for any control $u = (u_1, u_2, u_3) \in \mathcal{V}_{\text{ad}}$, the system~\eqref{E3.1} admits a unique nonnegative strong solution $\psi$.
Consequently, the cost functional $\mathcal{J}(\psi, u)$ is bounded on $\mathcal{V}_{\text{ad}}$. 
Therefore, there exists a constant
\begin{equation*}
    \chi := \inf_{u \in \mathcal{V}_{\text{ad}}} \mathcal{J}(\psi, u),
\end{equation*}
and a minimizing sequence $\bigl\{(\psi^n, u^n)\bigr\}_{n \in \mathbb{N}^*} \subset \mathbb{E}(\Omega) \times \mathcal{V}_{\text{ad}}$, such that
\begin{equation*}
     \chi = \lim_{n \to \infty} \mathcal{J}(\psi^n, u^n),
\end{equation*}
where $\psi^n$ is a solution to the system~\eqref{E2.1}--\eqref{E2.2}.

Furthermore, without loss of generality, we can assume that 
\begin{equation*}
   \forall n \in \mathbb{N}^*, \ \chi \leq \mathcal{J}(\psi^n, u^n) \leq \chi + \frac{1}{n}.
\end{equation*}
On one hand, since $\{u_1^n\}_{n\in\mathbb{N}^*}$, $\{u_2^n\}_{n\in\mathbb{N}^*}$, 
and $\{u_3^n\}_{n\in\mathbb{N}^*}$ are uniformly bounded in $L^2(\mathcal{U})$,
there exist functions $u_1^*, u_2^*, u_3^* \in L^2(\mathcal{U})$
and subsequences of $\{u_1^n\}_{n\in\mathbb{N}^*}$, $\{u_2^n\}_{n\in\mathbb{N}^*}$,
and $\{u_3^n\}_{n\in\mathbb{N}^*} $,
still denoted by themselves, respectively, such that 
\begin{equation}\label{E4.1}
        u_j^n \rightharpoonup u_j^* \quad \text{weakly in}\;L^2(\mathcal{U}),\;\forall j = 1, 2, 3.
\end{equation}
Moreover, since $\mathcal{V}_{\text{ad}}$ is a convex and closed subset of 
$\bigl(L^2(\mathcal{U})\bigr)^3$, it is also weakly closed. 
Hence, $u^* = (u_1^*, u_2^*, u_3^*)\in\mathcal{V}_{\text{ad}}$.

On the other hand, let $i\in \{1, 2, 3, 4, 5\} $. 
From Corollary~\ref{C1}, we have $\psi_i^n \in W^{1,2}(0, T; L^2(\Omega))$,
which implies $\psi_i^n \in C([0, T]; L^2(\Omega))$.
Meanwhile, by Proposition~\ref{P3}, there exists a constant $\mathscr{C}>0$, such that
\begin{equation}\label{E4.2}
   \Bigl\| \frac{\partial \psi_i^n}{\partial t} \Bigr\|_{L^2(\mathcal{U})}
   + \bigl\| \psi_i^n \bigr\|_{L^2(0, T; H^2(\Omega))}
   + \bigl\| \psi_i^n \bigr\|_{H^1(\Omega)}
   + \bigl\| \psi_i^n \bigr\|_{L^{\infty}(\mathcal{U})} \leq \mathscr{C}.
\end{equation}
By virtue of the inequality above, we can deduce the equicontinuity of the 
family $\{\psi_i^n\}_{n\in\mathbb{N}^*}$. 
Since $H^1(\Omega)$ is compactly embedded into $L^2(\Omega)$ (see \cite[Page 71]{Simon1986}), 
it follows that $\{\psi_i^n\}_{n\in\mathbb{N}^*}$ is relatively compact in $L^2(\Omega)$,
and we have $\|\psi_i^n(t)\|_{L^2(\Omega)} < \infty$ for all $t\in[0,T]$.
Then, by the Ascoli–Arzelà Theorem (see \cite[page 200]{Green1961}),
there exists a function $\psi_i^*\in C([0, T]; L^2(\Omega))$ 
and a subsequence of $\{\psi_i^n(t)\}_{n\in\mathbb{N}^*}$, still denoted by itself, such that
\begin{equation*}
    \lim_{n \to \infty} \sup_{t \in [0, T]} \bigl\| \psi_i^n(t) - \psi_i^*(t) \bigr\|_{L^2(\Omega)} = 0,
\end{equation*}
which means
\begin{equation}\label{E4.3}
    \psi_i^n \longrightarrow \psi_i^* \quad \text{uniformly for } t \in [0, T].
\end{equation}
From the inequality~\eqref{E4.2}, we know that the sequence $\bigl\{ \frac{\partial \psi_i^n}{\partial t} \bigr\}_{n\in\mathbb{N}^*}$ is bounded in $L^2(0, T; L^2(\Omega))$. 
Similarly, since $ \psi^n $ is a solution to \eqref{E2.1}--\eqref{E2.2}, 
it follows from~\eqref{E4.2} that $\{\Delta \psi_i^n\}_{n\in\mathbb{N}^*}$ is also bounded 
in $L^2(0, T; L^2(\Omega))$.
Therefore, there exists a subsequence of $\{\psi_i^n\}_{n\in\mathbb{N}^*}$, 
denoted again by $\{\psi_i^n\}_{n\in\mathbb{N}^*}$, such that
\begin{equation}\label{E4.4}
\begin{aligned}
   & \frac{\partial \psi_i^n}{\partial t} \rightharpoonup \frac{\partial \psi_i^*}{\partial t} \quad \text{weakly in } L^2(0, T; L^2(\Omega)), \\
   & \Delta \psi_i^n \rightharpoonup \Delta \psi_i^* \quad \text{weakly in } L^2(0, T; L^2(\Omega)), \\
   & \psi_i^n \longrightarrow \psi_i^* \quad \text{strongly in } L^\infty(0, T; H^1(\Omega)).
\end{aligned}
\end{equation}
Based on equations~\eqref{E4.1}--\eqref{E4.4}, we can establish the following convergences
\begin{align*}
   & \psi_1^n \psi_k^n \longrightarrow \psi_1^* \psi_k^* \quad \text{strongly in } L^\infty(0, T; H^1(\Omega)), \quad \forall\, k = 2, 3,\\
   & \psi_1^n u_1^n \rightharpoonup \psi_1^* u_1^* \quad \text{weakly in } L^2(0, T; L^2(\Omega)), \\
   & \psi_1^n \psi_k^n u_1^n \rightharpoonup \psi_1^* \psi_k^* u_1^* \quad \text{weakly in } L^2(0, T; L^2(\Omega)), \quad \forall\, k = 1, 2.
\end{align*}
Recall that $\nabla\psi_i^n \cdot \vec{n} = 0$ on $\Sigma_T$, 
which implies that $\nabla \psi_i^* \cdot \vec{n} = 0$ on $\Sigma_T$ as well.
Therefore, using the convergences above, we conclude that $\psi$ is a solution to the system~\eqref{E2.1}--\eqref{E2.2} corresponding to the control $u^*\in\mathcal{V}_{\text{ad}}$.
\end{proof}

\section{First-Order Necessary Conditions}\label{S5}
In this section, we derive the first-order necessary conditions for optimal control using the Lagrangian method. 
Let $P_S(t,x)$, $P_E(t,x)$, $P_I(t,x)$, $P_Q(t,x)$, and $P_R(t,x)$ be the adjoint variables corresponding to the state constraints respectively. 
The Lagrangian functional is given by
\begin{equation*}
   \mathcal{L}(\psi, u) = \int_0^T \int_\Omega \mathcal{A}_1(t, x) \, dx \, dt + \int_\Omega \mathcal{A}_2(T, x) \, dx + \sum_{i=1}^5 \int_0^T \int_\Omega P_{\psi_i} \Bigl[\frac{\partial \psi_i}{\partial t} + \mathcal{T}\psi_i - \mathcal{F}_i(\psi)\Bigr] \, dx \, dt.
\end{equation*}
Taking the variation of $\mathcal{L}$ with respect to each state variable and applying integration by parts along with Green’s formula yields 
for the $S$-equation (i.e., $\frac{\partial \mathcal{L}}{\partial S}(\psi^*, u^*) = 0$)
\begin{equation*}
\frac{\partial P_S}{\partial t} = -\lambda_S \Delta P_S + (\mu + u_1^* 
+ \beta_1(1-u_3^*)E^* + \beta_2(1-u_3^*)I^*)P_S - \beta_1(1-u_3^*)E^* P_E 
- \beta_2(1-u_3^*)I^* P_I - u_1^* P_R.
\end{equation*}
Similarly, for other state variables, we obtain
\begin{equation}\label{E5.1}
\left\{\begin{aligned}
\frac{\partial P_S}{\partial t} &= -\lambda_S \Delta P_S + 
(\mu + u_1^* + \beta_1(1-u_3^*)E^* + \beta_2(1-u_3^*)I^*)P_S \\
& \qquad- \beta_1(1-u_3^*)E^* P_E - \beta_2(1-u_3^*)I^* P_I - u_1^* P_R, \\
\frac{\partial P_E}{\partial t} &= -\lambda_E \Delta P_E + \beta_1(1-u_3^*)S^* P_S 
+ (\delta + \mu - \beta_1(1-u_3^*)S^*)P_E - \delta P_I + \kappa_1, \\
\frac{\partial P_I}{\partial t} &= -\lambda_I \Delta P_I + \beta_2(1-u_3^*)S^* P_S 
+ (\gamma + \mu - \beta_2(1-u_3^*)S^*)P_I - \gamma P_Q + \kappa_1, \\
\frac{\partial P_Q}{\partial t} &= -\lambda_Q \Delta P_Q - \rho P_S 
+ (\alpha + \rho + \mu + u_2^*)P_Q - (\alpha + u_2^*)P_R + \kappa_2, \\
\frac{\partial P_R}{\partial t} &= -\lambda_R \Delta P_R + \mu P_R.
\end{aligned}\right.
\quad \text{ in } \mathcal{U}, 
\end{equation}
Because the state variables are subject to homogeneous Neumann boundary conditions and the controls only act within the spatial domain, integration by parts yields no boundary terms. 
Consequently, the adjoint variables satisfy the same type of boundary conditions as the state variables: homogeneous Neumann conditions.
This means
\begin{equation}\label{E5.2}
    \nabla P_S \cdot \vec{n} = \nabla P_E \cdot \vec{n} = \nabla P_I \cdot \vec{n} = \nabla P_Q \cdot \vec{n} = \nabla P_R \cdot \vec{n} = 0, \ \text{on } \Sigma_T.
\end{equation}
The terminal conditions are obtained by taking the variation of the terminal cost term and setting the resulting expression equal to zero. Consequently, the adjoint variables are initialized backward in time, taking on values that reflect the sensitivity of the terminal cost to the final states
\begin{equation}\label{E5.3}
     P_S(T, x) = 0, \ P_E(T, x) = -\kappa_3, \ P_I(T, x) = -\kappa_3, \ P_Q(T, x) = -\kappa_4, \ P_R(T, x) = 0, \ \text{ in } \Omega.
\end{equation}

We can readily obtain the following results by introducing the change of variable, $\tau = T - t$ (time reversal), and applying the same arguments as in Corollary~\ref{C1}.
\begin{corollary}\label{C2}
Assume that $(\psi^*, u^*)$ is an optimal pair for the control system~\eqref{E2.1}--\eqref{E2.2}. 
Then, the adjoint system~\eqref{E5.1}--\eqref{E5.3} admits a unique strong solution $P_\psi$ such that
\begin{equation*}
    P_{\psi_i} \in W^{1,2}\bigl(0, T; L^2(\Omega)\bigr) \cap L^2\bigl(0, T; H^2(\Omega)\bigr) \cap L^\infty\bigl(0, T; H^1(\Omega)\bigr) \cap L^\infty(\mathcal{U}), 
    \ \forall i \in \{1, 2, 3, 4, 5\}.
\end{equation*}
\end{corollary}

A key step in our analysis to establish the first-order necessary conditions for optimal control is to prove the differentiability of the control-to-state mapping.\\
Let $u^{\varepsilon} = (u_1^{\varepsilon}, u_2^{\varepsilon}, u_3^{\varepsilon}) \in \mathcal{V}_{ad}$, and $(S^{\varepsilon}, E^{\varepsilon}, I^{\varepsilon}, Q^{\varepsilon}, R^{\varepsilon})$ be the corresponding solution to \eqref{E2.1}--\eqref{E2.2}. 
We define the \textit{control-to-state mapping}
\begin{equation}\label{E5.4}
   \Phi\colon u^{\varepsilon} \in\mathcal{V}_{ad} \subset \bigl(L^2(\mathcal{U})\bigr)^3 \longmapsto (S^{\varepsilon}, E^{\varepsilon}, I^{\varepsilon}, Q^{\varepsilon}, R^{\varepsilon}) \in \bigl(L^2(\mathcal{U})\bigr)^5.
\end{equation}
By Corollary~\ref{C1}, the mapping $\Phi$ is well-defined. We now state in Lemma~\ref{L1} the differentiability result.

\begin{lemma}\label{L1}
The mapping $\Phi$ defined in~\eqref{E5.4} is Gâteaux differentiable at $u^*$. Moreover, there exists a bounded linear operator
\begin{equation*}
   \Phi^{\prime}(u^*) \colon \widetilde{u} \in \bigl(L^2(\mathcal{U})\bigr)^3 \longmapsto (Y_S, Y_E, Y_I, Y_Q, Y_R) \in \bigl(L^2(\mathcal{U})\bigr)^5,
\end{equation*}
such that 
for every sufficiently small $\varepsilon > 0$ with $u^{\varepsilon} = u^* + \varepsilon \widetilde{u} \in \mathcal{V}_{ad}$, we have
\begin{equation*}
\lim_{\varepsilon \to 0} \biggl\| \frac{\Phi(u^{\varepsilon}) - \Phi(u^*)}{\varepsilon} - \Phi^{\prime}(u^*) \widetilde{u} \biggr\|_{\bigl(L^2(\mathcal{U})\bigr)^5} = 0,
\end{equation*}
where
\begin{equation*}
  \frac{\Phi(u^{\varepsilon}) - \Phi(u^*)}{\varepsilon} = 
  \Bigl(\frac{S^\varepsilon - S^*}{\varepsilon}, \frac{E^\varepsilon - E^*}{\varepsilon}, \frac{I^\varepsilon - I^*}{\varepsilon}, \frac{Q^\varepsilon - Q^*}{\varepsilon}, \frac{R^\varepsilon - R^*}{\varepsilon}\Bigr) =: Y^\varepsilon.
\end{equation*}
\end{lemma}

Before establishing the first-order necessary optimality conditions, we state the following proposition.

\begin{proposition}\label{P4}
Let the assumptions of Corollary~\ref{C1} hold, and let $\widetilde{u} \in \bigl(L^2(\mathcal{U})\bigr)^3$. 
Then the Gâteaux derivative of the control-to-state mapping at $u^*$, applied to $\widetilde{u}$, i.e.\ $\Phi^{\prime}(u^*) \widetilde{u}$, is the solution to the following linearized system
\begin{equation}\label{E5.5}
\left\{\begin{aligned}
   \frac{\partial Y_S}{\partial t}
    &= \lambda_S \Delta Y_S - \bigl(\mu + u_1^* + \beta_1(1-u_3^*)E^* + \beta_2(1-u_3^*)I^*\bigr) Y_S 
    - \beta_1(1-u_3^*)S^* Y_E\\
    & \qquad - \beta_2(1-u_3^*)S^* Y_I + \beta_1\widetilde{u}_3S^*E^* + \beta_2\widetilde{u}_3S^*I^* + \rho Y_Q - \widetilde{u}_1S^*,\\
\frac{\partial Y_E}{\partial t} &= \lambda_E \Delta Y_E + \beta_1(1-u_3^*)E^* Y_S 
    + \bigl[\beta_1(1-u_3^*)S^* - (\delta + \mu)\bigr]Y_E - \beta_1\widetilde{u}_3S^*E^*,\\
  \frac{\partial Y_I}{\partial t} &= \lambda_I \Delta Y_I + \beta_2(1-u_3^*)I^* Y_S + \delta Y_E 
  + \bigl[\beta_2(1-u_3^*)S^* - (\gamma + \mu)\bigr]Y_I\\
   & \qquad - \beta_2\widetilde{u}_3S^*I^*,\\
\frac{\partial Y_Q}{\partial t} &= \lambda_Q \Delta Y_Q + \gamma Y_I - (\alpha + \rho + \mu + u_2^*)Y_Q - \widetilde{u}_2Q^*,\\
\frac{\partial Y_R}{\partial t} &= \lambda_R \Delta Y_R + \alpha Y_Q + u_1^* Y_S + u_2^* Y_Q - \mu Y_R + \widetilde{u}_1S^* + \widetilde{u}_2Q^*,
\end{aligned}\right. \quad \text{ in } \mathcal{U}, 
\end{equation}
with
\begin{equation}\label{E5.6}
\left\{\begin{aligned}
&\nabla Y_S \cdot \vec{n} = \nabla Y_E \cdot \vec{n} = \nabla Y_I \cdot \vec{n} = \nabla Y_Q \cdot \vec{n} = \nabla Y_R \cdot \vec{n} = 0, & & \text{on } \Sigma_T,\\
&Y_S(0,x) = Y_E(0,x) = Y_I(0,x) = Y_Q(0,x) = Y_R(0,x) = 0,  & &\text{ in } \Omega.
\end{aligned}\right. 
\end{equation}
\end{proposition}

\begin{proof}[{Proof of Lemma~\ref{L1} and Proposition~\ref{P4}}]
To avoid making the paper longer, we briefly outline the main steps of the proof using the same methodology described in \cite[pages 21--25]{Zhou2024} and \cite[pages 456--458]{Zhou2019}. The proof proceeds as follows:

\textbf{Step 1} (\textit{Linearization via incremental quotient}).
For any perturbation $\widetilde{u} \in \bigl(L^2(\mathcal{U})\bigr)^3$, we define the incremental quotient
\begin{equation*}
Y^\varepsilon = \frac{\Phi(u^\varepsilon) - \Phi(u^*)}{\varepsilon}.
\end{equation*}

\textbf{Step 2} (\textit{Derivation of the linearized system}).
We show that the limit
\begin{equation*}
Y = \lim_{\varepsilon \to 0} Y^\varepsilon,
\end{equation*}
exists and satisfies a linearized system with zero initial data and homogeneous Neumann boundary conditions. This system depends linearly on $\widetilde{u}$ and the original solution $(S, E, I, Q, R)$.

\textbf{Step 3} (\textit{Existence and uniqueness}).
Using standard theory of linear parabolic systems, it is shown that this linearized system admits a unique strong solution.

\textbf{Step 4} (\textit{Convergence of the incremental quotients}).
It is then proven that
\begin{equation*}
    \lim_{\varepsilon \to 0} \| Y^\varepsilon - Y \|_{(L^2(\mathcal{U}))^5} = 0,
\end{equation*}
establishing the Gâteaux differentiability of $\Phi$, with $\Phi'(u^*) \widetilde{u} = Y$.

\textbf{Step 5} (\textit{Boundedness of the derivative}).
Finally, we confirm that $\Phi'(u^*)$ defines a bounded linear operator from $\bigl(L^2(\mathcal{U})\bigr)^3$ to $\bigl(L^2(\mathcal{U})\bigr)^5$.
\end{proof}

The following theorem establishes the first-order necessary conditions for optimality and provides a pointwise characterization of the optimal control in the absence of terminal penalization.

\begin{theorem}\label{T2}
Let $(S^*,E^*,I^*,Q^*,R^*, u_1^*,u_2^*,u_3^*)$ be an optimal pair for the control problem~\eqref{E2.1}--\eqref{E2.2} and $(P_S,P_E,P_I,P_Q,P_R)$ be the solution of the adjoint system~\eqref{E5.1}--\eqref{E5.3}. 
Then, for any $(u_1, u_2, u_3) \in \mathcal{V}_{ad}$, the following inequality holds
\begin{equation}\label{E5.7}
\begin{aligned}
  \int_0^T \int_\Omega & (S^* P_S - S^* P_R + w_1) (u_1 - u_1^*) \,dx dt
      + \int_0^T \int_\Omega (Q^* P_Q - Q^* P_R + w_2) (u_2 - u_2^*) \,dx dt\\
  &+ \int_0^T \int_\Omega \bigl(\beta_1 S^*E^* (P_E - P_S) + \beta_2S^*I^* (P_I - P_S) + w_3\bigr) (u_3 - u_3^*) \,dx dt\\
\geq& - \int_\Omega \sigma_1 (u_1 - u_1^*)(T,x) \, dx - \int_\Omega  \sigma_2 (u_2 - u_2^*)(T,x) \, dx - \int_\Omega \sigma_3 (u_3 - u_3^*)(T,x) \, dx.
\end{aligned}
\end{equation}
Furthermore, if $\sigma_1 = \sigma_2 = \sigma_3 \equiv 0$, then the optimal control $(u_1^*,u_2^*,u_3^*)$ can be characterized point-wise by
\begin{equation}\label{E5.8}
\begin{aligned}
u_1^*(t,x) &= \begin{cases}
1, & \text{if } (S^*P_S - S^*P_R + w_1)(t,x) \leq 0,\\
0, & \text{otherwise},
\end{cases}\\
u_2^*(t,x) &= \begin{cases}
1, & \text{if } (Q^*P_Q - Q^*P_R + w_2)(t,x) \leq 0,\\
0, & \text{otherwise},
\end{cases}\\
u_3^*(t,x) &= \begin{cases}
1, & \text{if } (\beta_1S^*E^*(P_E - P_S) + \beta_2S^*I^*(P_I - P_S) + w_3)(t,x) \leq 0,\\
0, & \text{otherwise}.
\end{cases}
\end{aligned}
\end{equation}
\end{theorem}

The proof of this theorem is deferred to Appendix~\ref{App1} to maintain the flow of this section.


\section{Numerical Results}\label{S6}
This section aims to investigate the dynamics of the proposed SEIQR reaction-diffusion model with optimal control strategies using numerical methods.
Specifically, we will evaluate the spatial-temporal evolution of the disease in a two-dimensional urban domain under various intervention scenarios, including the presence or absence of vaccination, treatment, and social distancing.
We place special emphasis on quantifying the impact of spatial diffusion and control measures on infection containment over a finite time horizon.

We consider the square spatial domain $\Omega = [0,50] \times [0,50]$ km$^2$, representing a densely populated urban area. The domain is discretized into $1 \text{ km} \times 1 \text{ km}$ grid cells.
At the initial time, $t = 1$, we assume the population is uniformly distributed with 100 susceptible individuals per cell across $\Omega$ (i.e., 100\,\% of the population), except in a central subregion $\Omega_c = \text{cell}(15,15)$, which represents a city center, for example.
In this subdomain, 15\,\% of the local population is initialized as exposed, 10\% as infected, and the remaining 75\,\% as susceptible. The simulation period spans 60 days.

The spatial diffusion of the disease is driven by the movement of individuals and is modeled through diffusion terms in the SEIQR system.
We use a finite difference method (FDM) for spatial discretization and an implicit time-stepping scheme to ensure stability.
We solve the optimal control problem iteratively using a forward–backward sweep algorithm, updating the control functions at each iteration according to the characterizations derived in Section~\ref{S5}.
Using a forward approach in time solves the state problem, while a backward approach solves the dual system according to the transversality conditions.

We now define the parameter values that govern the dynamics of the SEIQR reaction-diffusion system.
Since there is no empirical data available for this specific setting, we have adopted biologically plausible parameter values capable of capturing the qualitative behavior of disease transmission and control. Each parameter is chosen to reflect commonly observed epidemiological patterns and intervention mechanisms, ensuring consistency with real-world expectations.

\begin{itemize}
\item In this model, the recruitment rate $\Lambda$ represents the influx of new individuals into the population through birth or immigration.
The natural death rate  $\mu$ is assumed to be small and constant because it reflects standard demographic attrition unrelated to the disease.

\item Regarding disease transmission, the model distinguishes two primary pathways:
The first is the interaction between susceptible and exposed individuals, governed by $\beta_1$, and the second is the interaction between susceptible and infected individuals, represented by $\beta_2$.
Biologically, it is expected that $\beta_2 > \beta_1$, since symptomatic infectious individuals are more likely to transmit the disease than those in the latent phase.
This reflects the increased viral shedding and higher contact rates associated with the symptomatic stage.

\item The progression from the exposed state to the infectious state occurs at a rate of $\delta$, depending on the disease's incubation period.
Once individuals are infectious, they may be detected and quarantined at a rate of $\gamma$, 
which is influenced by testing capacity and contact tracing efficiency.
Usually, $\gamma$ is smaller than $\delta$ because quarantine depends on public health interventions, while disease progression is a natural biological process.

\item Quarantined individuals can either recover at rate $\alpha$, or, if not infected,
return to the susceptible pool at rate $\rho$. 
Generally, $\alpha$ exceeds $\rho$, as recovery from illness is more likely to occur within quarantine than reclassification of individuals who were mistakenly quarantined as susceptible.
\end{itemize}
These inter-parameter relationships, which are grounded in biological reasoning and public health practice, form the basis of the numerical simulations. 
They allow us to investigate the spatiotemporal evolution of the epidemic under different intervention strategies.
Table~\ref{Tab3} summarizes the parameter values used in our simulations.

\begin{table}[H]
\centering
\setlength{\tabcolsep}{0.8cm}
\caption{Parameter values for the SEIQR model~\eqref{E2.1}.}\label{Tab3}
\adjustbox{max width=\textwidth}{
\begin{tabular}{ccc||ccc}
\hline 
\textbf{Symbol} & \textbf{Value} & \textbf{Source} & \textbf{Symbol} & \textbf{Value} & \textbf{Source} \\
\hline \hline 
$\Lambda$ & 1 & \cite{Khanh2016} & $\delta$ & 0.05 & \cite{Ozalp2011}\\
\hline
$\beta_1$ & 0.06 & Assumed & $\gamma$ & 0.02 & \cite{LemosPaio2022}\\
\hline
$\beta_2$ & 0.07 & Assumed & $\alpha$ & 0.05 & \cite{Paul2022}\\
\hline
$\mu$ & 0.01 & \cite{Kumar2023} & $\rho$ & 0.01 & \cite{Khanh2016}\\
\hline
$\lambda_S, \lambda_E, \lambda_I, \lambda_R$ & 0.1 & \cite{SidiAmmi2022} & $\lambda_Q$ & 0.001 & Assumed\\
\hline
\end{tabular}
}
\end{table}

In real-world scenarios, quarantined individuals are expected to remain in fixed locations, such as hospitals or isolation facilities. This results in negligible spatial movement.
We therefore assume that the diffusion coefficient for the quarantine compartment is $\lambda_Q = 0.001$, while all other compartments are assigned a diffusion coefficient of 0.1 to reflect the moderate mobility of individuals in the general population.

To evaluate the effectiveness of various intervention strategies in controlling the epidemic, we analyze eight scenarios that incorporate different combinations of control measures, $u_1$, $u_2$, and $u_3$.
These controls represent key pharmaceutical interventions (PIs) and non-pharmaceutical interventions (NPIs), as detailed in Table~\ref{Tab2}. 
Specifically, $u_1$ corresponds to vaccination of susceptible and recovered individuals to increase population immunity, $u_2$ models treatment of quarantined individuals to accelerate recovery, and $u_3$ denotes social distancing and public awareness measures that reduce effective contact rates among susceptible, exposed, and infected individuals.
We simulate the epidemic dynamics under the following eight cases:
\begin{itemize}
\item[-] \textbf{Case 1:} no control ($u_1 = u_2 = u_3 = 0$); baseline scenario.
\item[-] \textbf{Case 2:} only vaccination ($u_1$ active).
\item[-] \textbf{Case 3:} only treatment ($u_2$ active).
\item[-] \textbf{Case 4:} only social distancing ($u_3$ active).
\item[-] \textbf{Case 5:} vaccination and treatment ($u_1$ and $u_2$ active).
\item[-] \textbf{Case 6:} vaccination and social distancing ($u_1$ and $u_3$ active).
\item[-] \textbf{Case 7:} treatment and social distancing ($u_2$ and $u_3$ active).
\item[-] \textbf{Case 8:} all controls are active ($u_1$, $u_2$, and $u_3$).
\end{itemize}
We evaluate how each control, individually or in combination, influences infection prevalence, recovery dynamics, and overall system behavior as governed by the model~\eqref{E2.1}–\eqref{E2.2}.  by comparing these cases.
This analysis reveals the impact of vaccination, treatment, and social distancing on reducing disease spread and improving epidemic outcomes.

The results below present a comparative analysis of disease progression with and without optimal control interventions, demonstrating the role of each control component in minimizing both the infection spread and the total cost functional $\mathcal{J}$.

Figure~\ref{F3} presents the evolution of the SEIQR system in the absence of any intervention (i.e., $u_1 = u_2 = u_3 = 0$). The epidemic spreads rapidly across the spatial domain, 
 reaching high peak levels of infection. 
 Without vaccination, treatment, or quarantine, the disease grows uncontrollably over time and space. This scenario leads to the highest value of $\mathcal{J}$, as shown in Table~\ref{Tab4}.

\begin{figure}[H]
\centering
\includegraphics[width=1\textwidth]{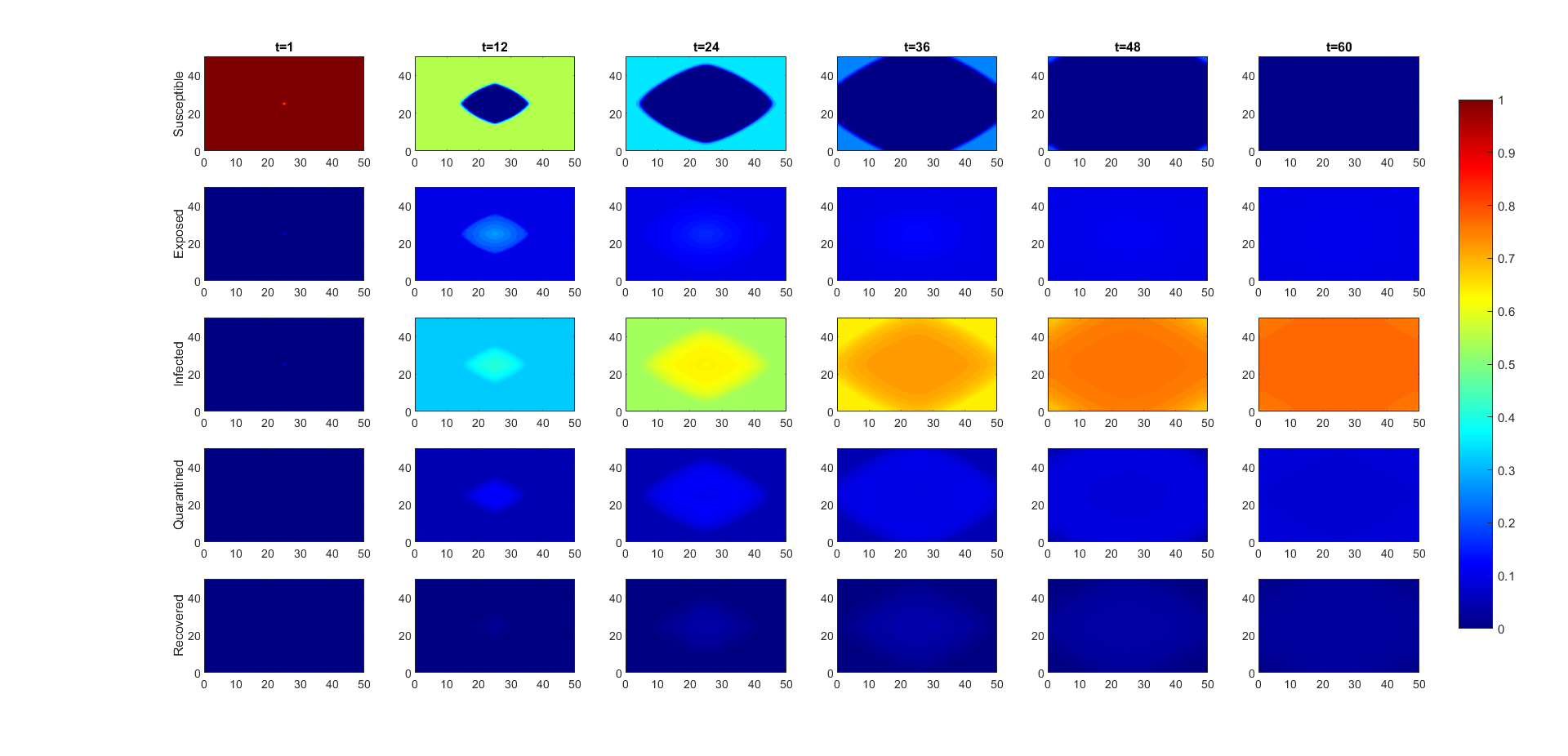}
\caption{Spatiotemporal evolution of the SEIQR model without any control interventions.}\label{F3}
\end{figure}

Figure~\ref{F4} illustrates the scenario in which only the vaccination strategy ($u_1 \neq 0$, $u_2 = u_3 = 0$) is applied. 
Vaccination reduces the susceptible population, which slows the emergence of new exposed individuals.
However, since neither treatment nor social distancing is used, infected individuals continue to spread the disease.
As shown in this figure, the spatial extent of infection is smaller than in Figure~\ref{F3}, and the total cost $\mathcal{J}$ is moderately reduced.

\begin{figure}[H]
\centering
\includegraphics[width=1\textwidth]{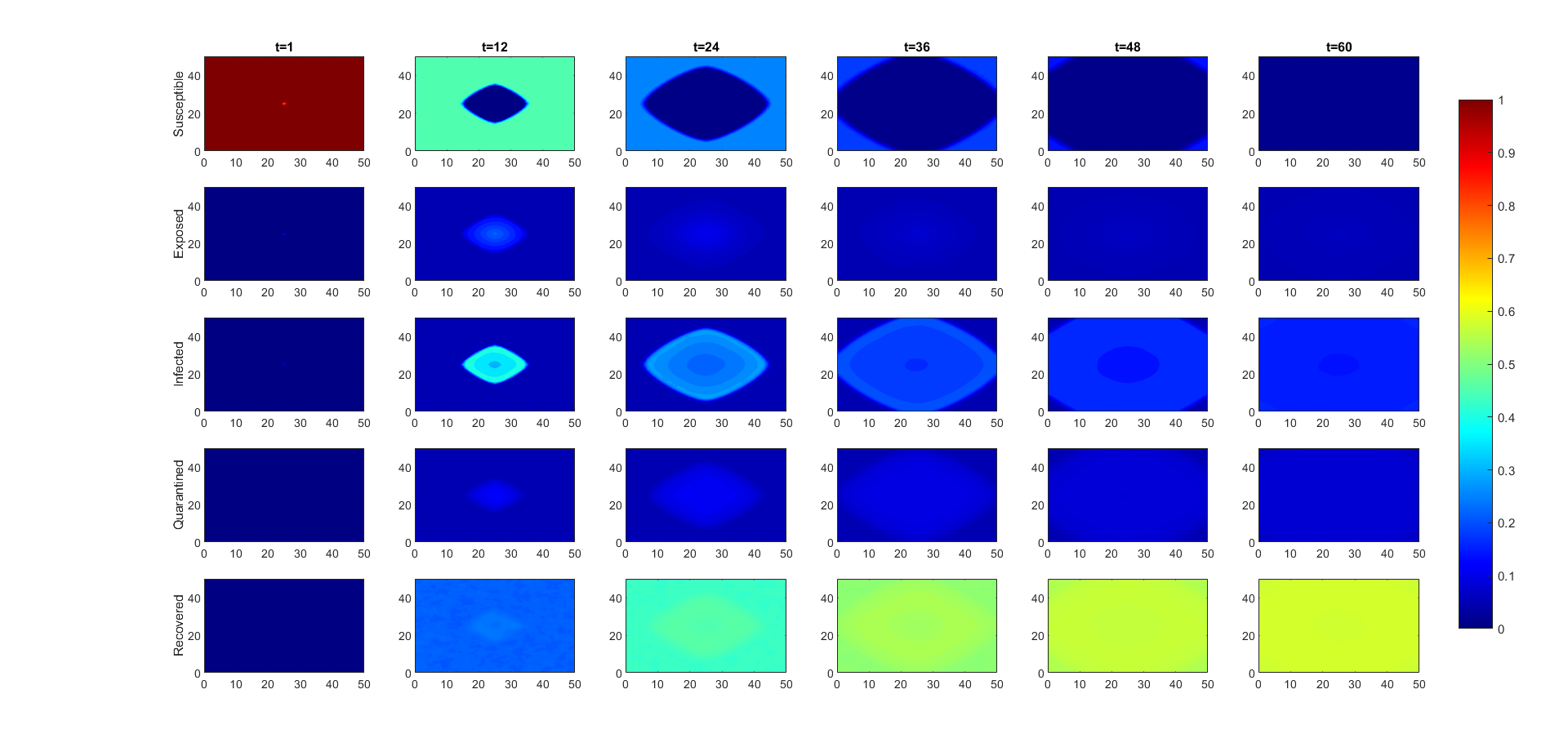}
\caption{Effect of the vaccination strategy (only $u_1$) on the spatiotemporal dynamics of the SEIQR model.}\label{F4}
\end{figure}

Figure~\ref{F5} shows the implementation of the quarantine treatment only 
 ($u_2 \neq 0$, $u_1 = u_3 = 0$). 
This scenario reduces the number of infected individuals but fails to prevent new exposures. Compared to Figure~\ref{F3}, the infected zones are less intense, and the cost functional is lower; however, the epidemic persists spatially and temporally.

\begin{figure}[H]
\centering
\includegraphics[width=1\textwidth]{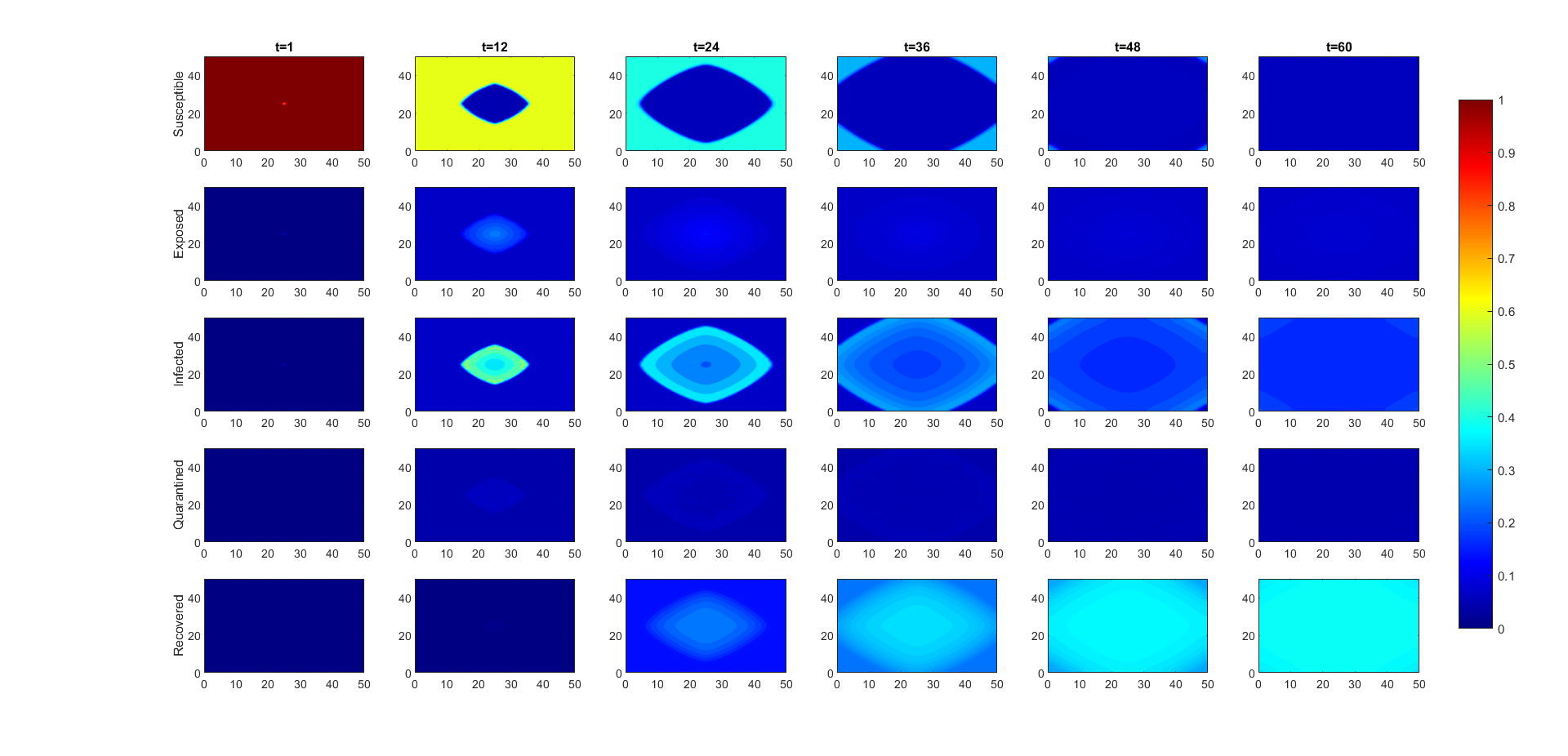}
\caption{Effect of the treatment strategy (only $u_2$) on the spatiotemporal dynamics of the SEIQR model.}\label{F5}
\end{figure}

Figure~\ref{F6} shows the impact of applying social distancing exclusively to the exposed and infected individuals ($u_3 \neq 0$, $u_1 = u_2 = 0$). 
As shown in the figure, public awareness is an effective strategy for limiting progression to exposed and infected rooms, thereby reducing overall disease spread.
The epidemic remains more spatially contained, and the value of the objective functional decreases significantly compared to the no-control scenario shown in Figure~\ref{F3}.

\begin{figure}[H]
\centering
\includegraphics[width=1\textwidth]{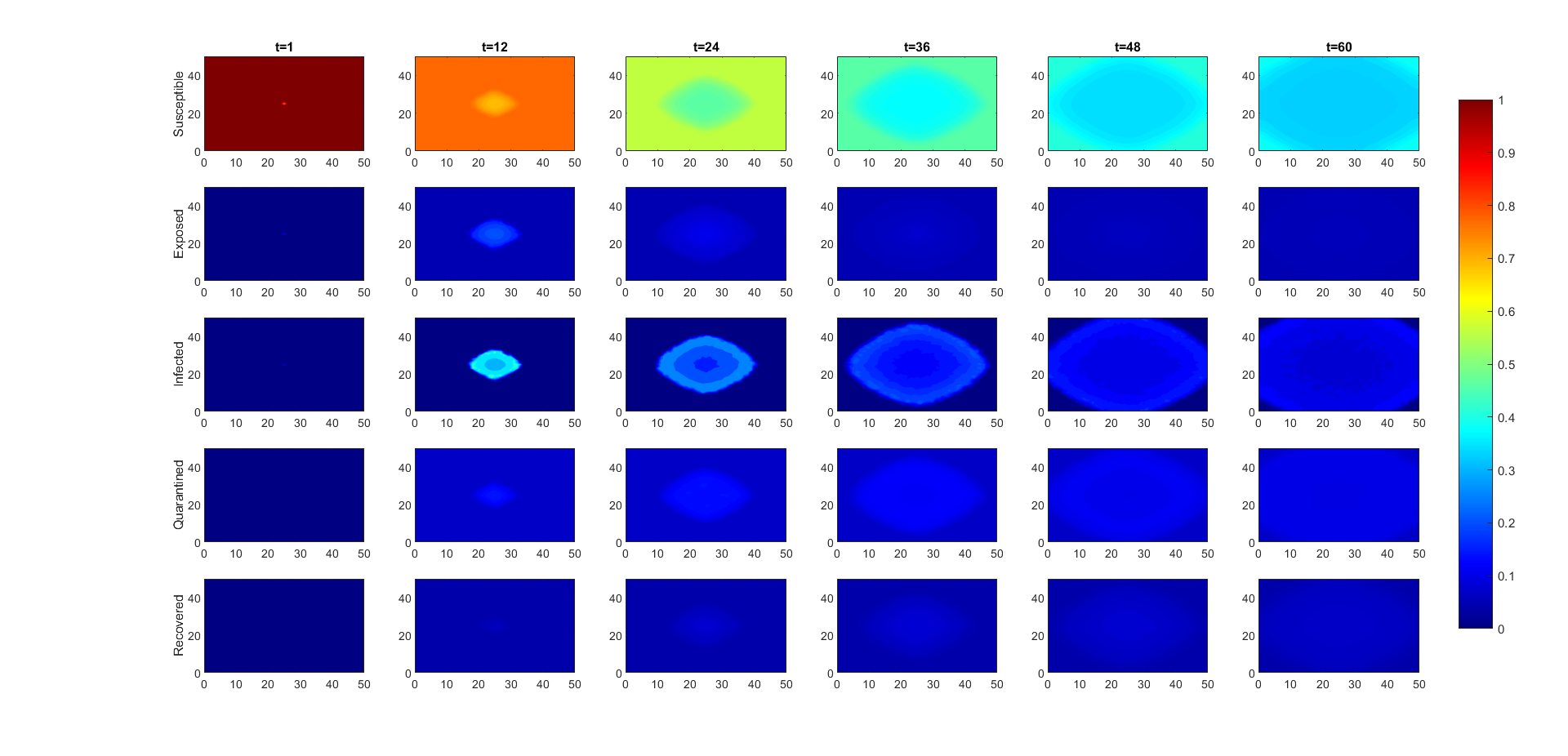}
\caption{Effect of the social distancing strategy (only $u_3$) on the spatiotemporal dynamics of the SEIQR model.}\label{F6}
\end{figure}

Figure~\ref{F7} presents the combined impact of vaccination and treatment ($u_1, u_2 \neq 0$, $u_3 = 0$). 
This combination effectively reduces both the number of susceptible and quarantined individuals, and shortens the duration of infection. 
As observed in this case, the disease spread is more limited than in the scenarios shown in Figures~\ref{F4} and \ref{F5}, in which only one intervention was applied.
Among the partial control strategies, this case exhibits one of the lowest cost functional $\mathcal{J}$ values, as shown in Table~\ref{Tab4}.

\begin{figure}[H]
\centering
\includegraphics[width=1\textwidth]{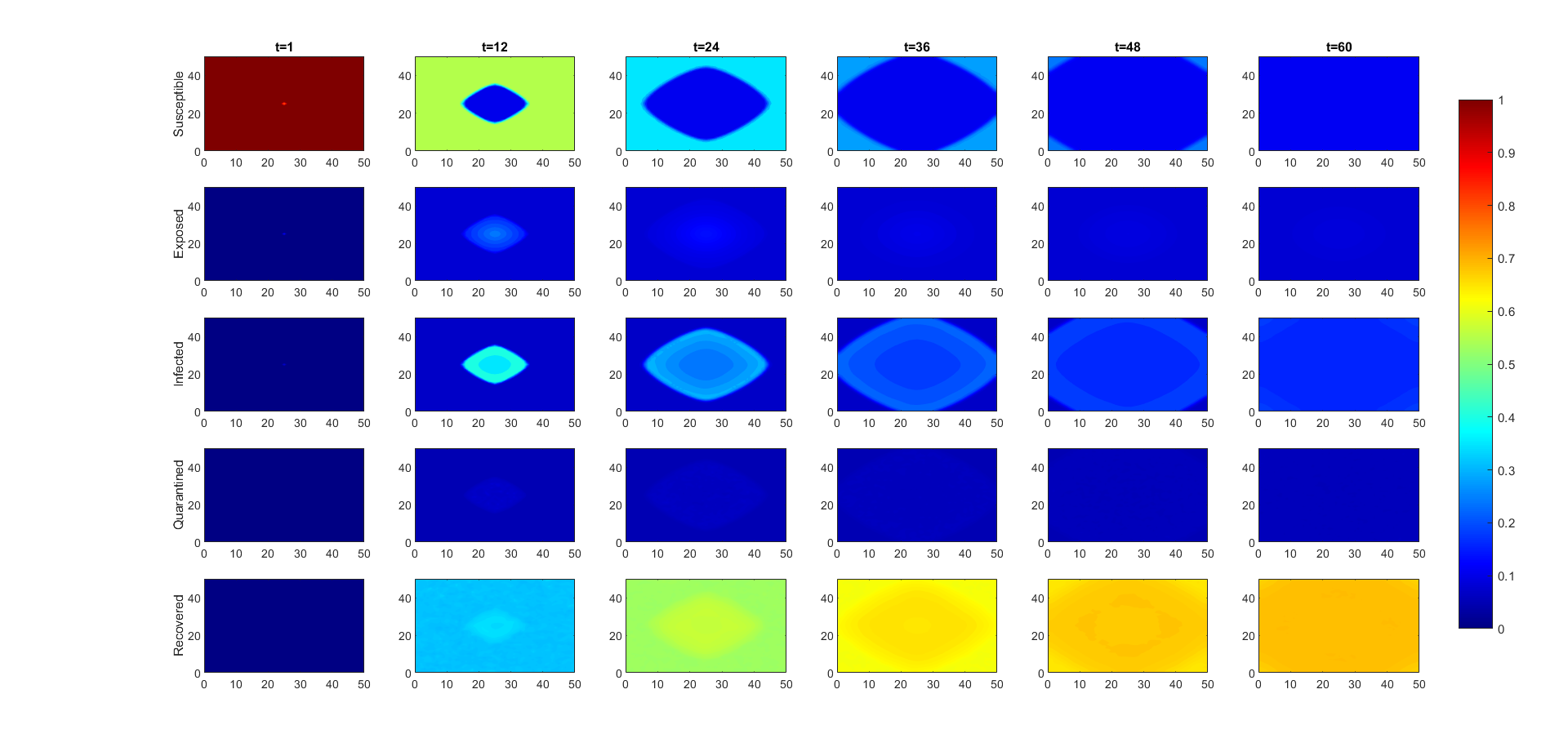}
\caption{Impact of combined vaccination and treatment strategies ($u_1$ and $u_2$) on the spatiotemporal evolution of the SEIQR model.}\label{F7}
\end{figure}

Figure~\ref{F8} shows the joint application of vaccination and social distancing ($u_1, u_3 \neq 0$, $u_2 = 0$). 
This strategy is particularly effective because it prevents new infections. 
The infected population is significantly reduced and the outbreak is largely suppressed.

\begin{figure}[H]
\centering
\includegraphics[width=1\textwidth]{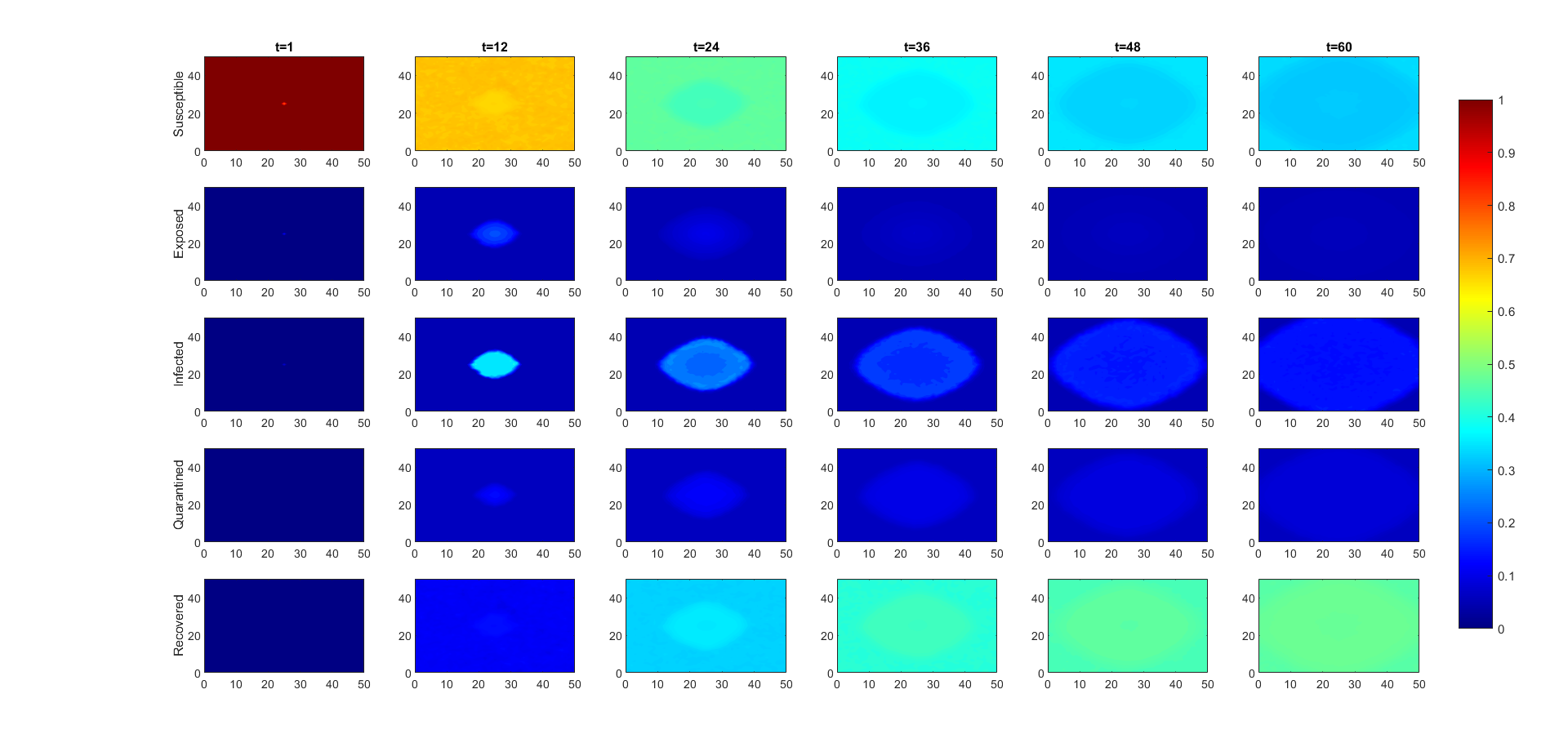}
\caption{Impact of combined vaccination and social distancing strategies ($u_1$ and $u_3$) on the spatiotemporal evolution of the SEIQR model.}\label{F8}
\end{figure}

Figure~\ref{F9} demonstrates the combined use of treatment and social distancing ($u_2, u_3 \neq 0$, $u_1 = 0$). 
Although vaccination is not employed, progression from the exposed to the infected state is blocked and infected individuals are treated.
As shown in this figure, the epidemic is mitigated, with reduced spatial propagation. 
This combination of controls yields a low value of $\mathcal{J}$, though not as low as in cases that include vaccination, like Figure~\ref{F7}.

\begin{figure}[H]
\centering
\includegraphics[width=1\textwidth]{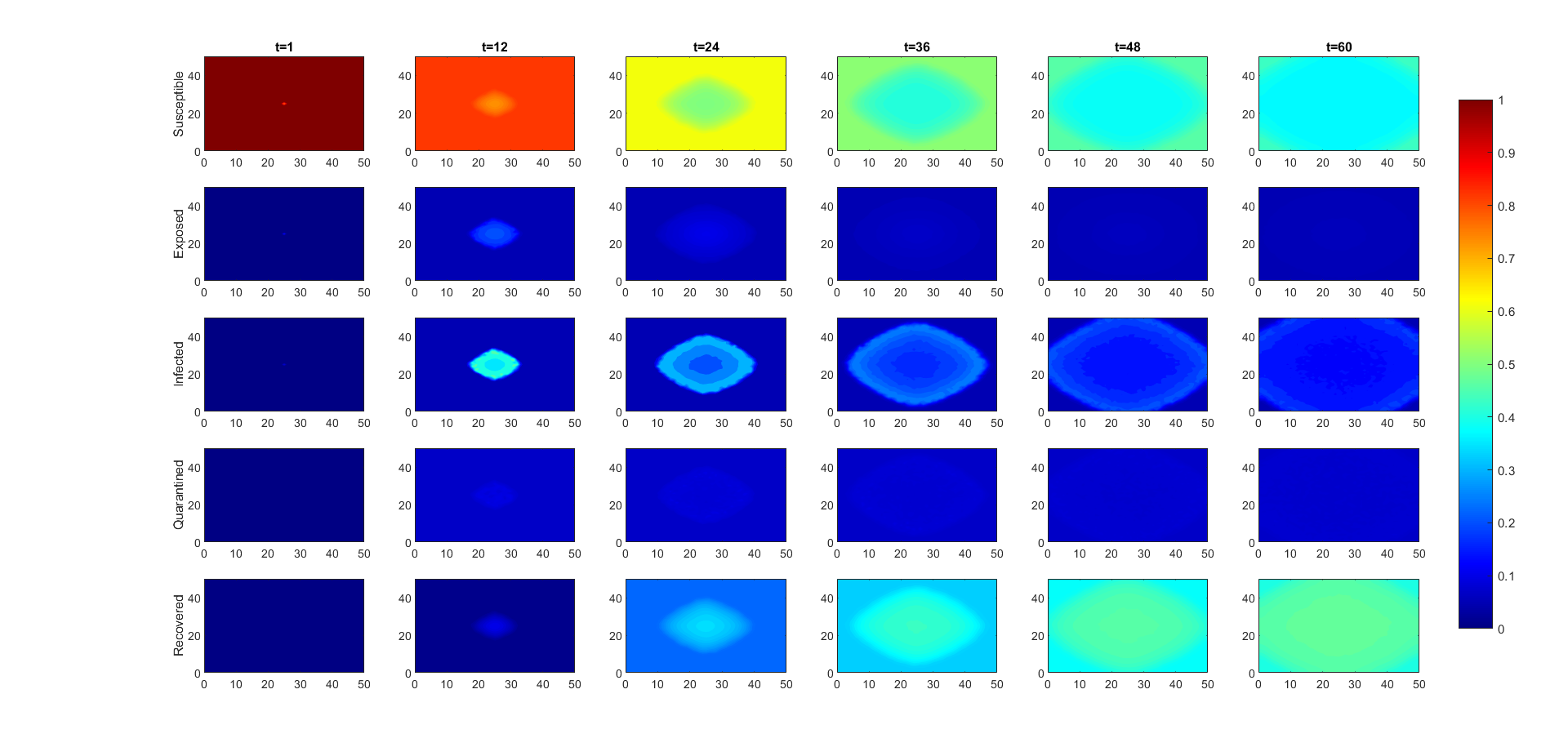}
\caption{Impact of combined treatment and social distancing strategies ($u_2$ and $u_3$) on the spatiotemporal evolution of the SEIQR model.}\label{F9}
\end{figure}

Figure~\ref{F10} shows that when all three controls are active ($u_1, u_2, u_3 \neq 0$), 
the most comprehensive approach is achieved. 
The results confirm that this strategy is the most effective.
The infection has nearly been eradicated and the spatial spread is minimal.
This leads to the lowest cost functional $\mathcal{J}$ among all scenarios,
demonstrating the importance of combining multiple interventions.

\begin{figure}[H]
\centering
\includegraphics[width=1\textwidth]{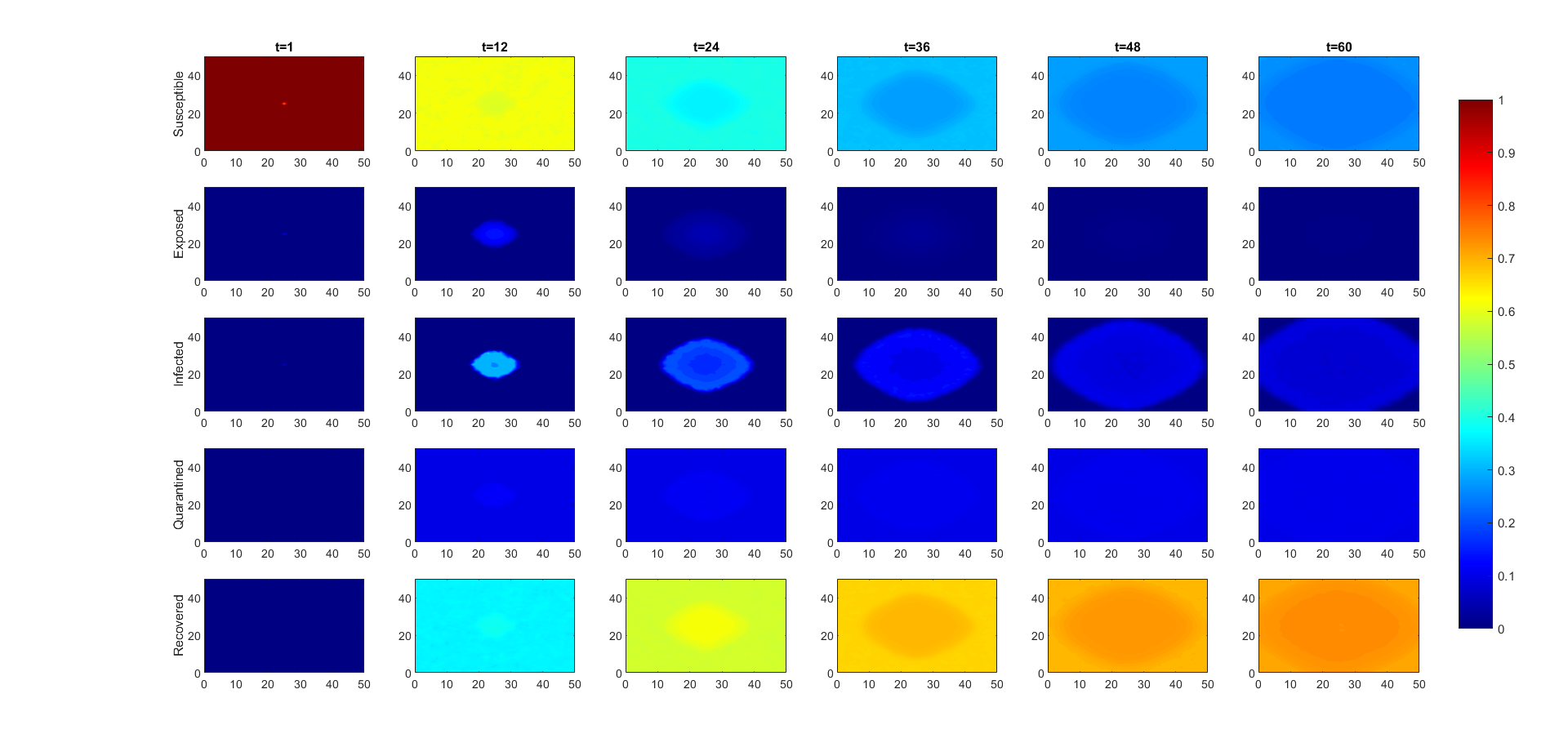}
\caption{Effect of combined vaccination, treatment, and social distancing strategies on the spatiotemporal dynamics of the SEIQR model.}\label{F10}
\end{figure}

Table~\ref{Tab4} provides a comparison of the cost functional $\mathcal{J}$ across all eight cases presented in Figures~\ref{F3}--\ref{F10}.
The results confirm that the full control case is optimal, 
while the absence of control results in the highest cost.

\begin{table}[H]
\centering
\caption{Cost functional $\mathcal{J}$ for various control strategies in the SEIQR model.}\label{Tab4}
\adjustbox{max width=\textwidth}{
\begin{tabular}{c||cccccccc}
\hline
\textbf{Cases} & (1) & (2) & (3) & (4) & (5) & (6) & (7) & (8)\\ 
\hline\hline
\textbf{Value of} & \multirow{2}{*}{$5.7377e^{+03}$} & \multirow{2}{*}{$2.2086e^{+03}$} & \multirow{2}{*}{$2.3499e^{+03}$} & \multirow{2}{*}{$2.5903e^{+03}$} & \multirow{2}{*}{$1.4402e^{+03}$} & \multirow{2}{*}{$1.6186e^{+03}$} & \multirow{2}{*}{$1.6632e^{+03}$} & \multirow{2}{*}{$1.3562e^{+03}$} \\ 
\textbf{$\mathcal{J}(\psi, u)$} &&&&&&&&\\
\hline
\end{tabular}
}
\end{table}

As a result, the ordering from the largest to the smallest value of the objective functional $\mathcal{J}$, along with the corresponding effectiveness of each strategy in controlling the spread of the disease, is as follows:
\begin{itemize}
\item[-] \textbf{Case 1:} In the absence of any intervention, the disease spreads freely, 
resulting in the highest epidemiological and economic cost.

\item[-] \textbf{Case 4:} Implementing distancing for exposed and infected individuals slightly mitigates transmission, but alone, it is insufficient to significantly control the outbreak.

\item[-] \textbf{Case 3:} Treatment reduces the infectious population but does not prevent new exposures.

\item[-] \textbf{Case 2:} Vaccination protects susceptible individuals, yielding a better outcome than treatment or distancing alone.

\item[-] \textbf{Case 7:} Combining reactive measures (treatment and distancing)
is more effective than either strategy individually, yet it still lacks prevention at the susceptible level.

\item[-] \textbf{Case 6:} Adding a preventive measure, such as like vaccination 
greatly reduces disease spread when combined with distancing.

\item[-] \textbf{Case 5:} This pairing (vaccination and  treatment) effectively tackles both susceptibility and infectiousness.

\item[-] \textbf{Case 8:} Applying all three controls simultaneously yields the most significant reduction in disease prevalence and control costs, resulting in the lowest value of the objective functional $\mathcal{J}$.
\end{itemize}

\section{Conclusion and Perspectives}\label{S7}
In this work, we developed and analyzed a reaction-diffusion SEIQR epidemic model that incorporates spatial heterogeneity and three time-dependent control strategies: vaccination, treatment, and social distancing.
Our primary objective was to investigate how these interventions, both individually and in combination, influence the spatiotemporal spread of an infectious disease.

First, we established the well-posedness of the model by proving the existence, uniqueness, positivity, and boundedness of the strong solution using semigroup theory.
Then, we demonstrated the existence of an optimal solution and derived the first-order necessary optimality conditions associated with the state model using a variational approach.

To validate and illustrate these findings, we conducted a series of numerical simulations over a spatially discretized urban domain.
The results show that all three interventions significantly reduce the burden of the epidemic, and combining them yields the best outcome.
Of the scenarios considered, the full control strategy was the most effective in minimizing both infection prevalence and the cost functional.

This study underscores the importance of integrating pharmaceutical and non-pharmaceutical interventions in spatial-temporal epidemic control.
Several promising directions can be pursued to enhance the model's applicability.
One natural extension is incorporating time delays to capture incubation periods, behavioral response lags, or delays in implementing control measures.
The model could also be enriched by accounting for mobility effects, vaccination waning, or imperfect efficacy-factors that better reflect real-world dynamics.
From a methodological perspective, introducing stochastic elements or quantifying parameter uncertainty would increase the robustness of the control strategies.
Finally, future research could explore developing data-driven feedback control laws and using machine learning techniques to calibrate model parameters and support real-time decision-making during outbreaks.

\begin{appendices}

\section{Proof of the First-Order Necessary Conditions}\label{App1}
Here we provide a detailed proof of Theorem~\ref{T2} presented in Section~\ref{S5} of the main text.

\begin{proof}[{Proof of Theorem~\ref{T2}}]
Let us assume that $(\psi^*, u^*)$ is an optimal pair for~\eqref{E2.1}--\eqref{E2.2}. Then, for all $\varepsilon > 0$, we have
\begin{equation*}
\mathcal{J}(\psi^*, u^*) \leq \mathcal{J}(\psi^{\varepsilon}, u^{\varepsilon}),
\end{equation*}
where $u^{\varepsilon}$ and $\psi^{\varepsilon}$ are defined in Lemma~\ref{L1}. 
For any $\widetilde{u} \in \bigl(L^2(\mathcal{U})\bigr)^3$, we have
\begin{equation*}
\begin{aligned}
   0 \leq& \frac{\mathcal{J}(\psi^{\varepsilon}, u^{\varepsilon}) - \mathcal{J}(\psi^*, u^*)}{\varepsilon}\\
   =& \frac{1}{\varepsilon} \bigg( \int_0^T \int_\Omega (\kappa_1 E^\varepsilon + \kappa_1 I^\varepsilon + \kappa_2 Q^\varepsilon + w_1 u_1^\varepsilon + w_2 u_2^\varepsilon + w_3 u_3^\varepsilon )\, dx \, dt\\
   &+ \int_\Omega \bigl( \kappa_3 E^\varepsilon(T,x) + \kappa_3 I^\varepsilon(T,x) + \kappa_4 Q^\varepsilon(T,x) + \sigma_1 u_1^\varepsilon(T,x) + \sigma_2 u_2^\varepsilon(T,x) + \sigma_3 u_3^\varepsilon(T,x) \bigr)\, dx\\
   &- \int_0^T \int_\Omega (\kappa_1 E^* + \kappa_1 I^* + \kappa_2 Q^* + w_1 u_1^* + w_2 u_2^* + w_3 u_3^*)\, dx \, dt\\
   &- \int_\Omega \bigl( \kappa_3 E(T,x) + \kappa_3 I(T,x) + \kappa_4 Q(T,x) + \sigma_1 u_1(T,x) + \sigma_2 u_2(T,x) + \sigma_3 u_3(T,x) \bigr)\, dx \bigg).
\end{aligned}
\end{equation*}
Afterwards,
\begin{equation}\label{E5.9}
\begin{aligned}
    0 \leq&  \int_0^T \int_\Omega (\kappa_1 Y_E^\varepsilon + \kappa_1 Y_I^\varepsilon + \kappa_2 Y_Q^\varepsilon + w_1 \widetilde{u}_1 + w_2 \widetilde{u}_2 + w_3 \widetilde{u}_3 ) \,dx \, dt\\
     &+ \int_\Omega \bigl( \kappa_3 Y_E^\varepsilon(T,x) + \kappa_3 Y_I^\varepsilon(T,x) + \kappa_4 Y_Q^\varepsilon(T,x) + \sigma_1 \widetilde{u}_1(T,x) + \sigma_2 \widetilde{u}_2(T,x) + \sigma_3 \widetilde{u}_3(T,x) \bigr)\, dx.
\end{aligned}
\end{equation}
On the one hand, it follows from Proposition~\ref{P4} that
\begin{equation*}
    Y_E^{\varepsilon} \to Y_E, \ Y_I^{\varepsilon} \to Y_I, \ Y_Q^{\varepsilon} \to Y_Q, \ 
     \text{ in } \ L^2(\mathcal{U}) \text { as } \varepsilon \to 0.
\end{equation*}
Since $L^2(\mathcal{U}) \subset L^1(\mathcal{U})$, then, we have
\begin{equation}\label{E5.10}
    Y_E^{\varepsilon} \to Y_E, \; Y_I^{\varepsilon} \to Y_I, \; Y_Q^{\varepsilon} \to Y_Q, \ \text { in } L^1(\mathcal{U}) \text { as } \varepsilon \to 0.
\end{equation}
On the other hand, since the coefficients in the system associated with $Y^\varepsilon$ are uniformly bounded, we can apply the same argument from Corollary~\ref{C1} to this linear parabolic system. 
Consequently, there exists a constant $C>0$, independent of $\varepsilon$, such that
\begin{equation}\label{E5.11}
\begin{aligned}
& \Bigl\|\frac{\partial Y_E^\varepsilon}{\partial t}\Bigr\|_{L^2(\mathcal{U})}
+\bigl\|Y_E^\varepsilon\bigr\|_{L^2(0, T; H^2(\Omega))}
+\bigl\|Y_E^\varepsilon\bigr\|_{H^1(\Omega)} 
+\bigl\|Y_E^\varepsilon\bigr\|_{L^{\infty}(\mathcal{U})} \leq C,\\
& \Bigl\|\frac{\partial Y_I^\varepsilon}{\partial t}\Bigr\|_{L^2(\mathcal{U})}
+\bigl\|Y_I^\varepsilon\bigr\|_{L^2(0, T; H^2(\Omega))}
+\bigl\|Y_I^\varepsilon\bigr\|_{H^1(\Omega)} 
+\bigl\|Y_I^\varepsilon\bigr\|_{L^{\infty}(\mathcal{U})} \leq C,\\
& \Bigl\|\frac{\partial Y_Q^\varepsilon}{\partial t}\Bigr\|_{L^2(\mathcal{U})}
+\bigl\|Y_Q^\varepsilon\bigr\|_{L^2(0, T; H^2(\Omega))}
+\bigl\|Y_Q^\varepsilon\bigr\|_{H^1(\Omega)} 
+\bigl\|Y_Q^\varepsilon\bigr\|_{L^{\infty}(\mathcal{U})} \leq C.
\end{aligned}
\end{equation}
This implies that the family of functions $\bigl\{(Y_S^{\varepsilon}, Y_E^{\varepsilon}, Y_I^{\varepsilon}, Y_Q^{\varepsilon}, Y_R^{\varepsilon} )\bigr\}_{\varepsilon}$ is equicontinuous in $\bigl(L^2(\mathcal{U})\bigr)^5$ for $t \in [0, T]$, 
with respect to the arbitrary positive number $\varepsilon$.
Furthermore, since $H^1(\Omega)$ is compactly embedded in $L^2(\Omega)$ (see, e.g.\ \cite{Martin1990}), it follows from~\eqref{E5.11} that there exists a constant $C_2 > 0$ such that $\|Y_E^{\varepsilon}\|_{L^2(\Omega)} \leq C_2$, 
$\|Y_I^{\varepsilon}\|_{L^2(\Omega)} \leq C_2$, and $\|Y_Q^{\varepsilon}\|_{L^2(\Omega)} \leq C_2$ 
for any $t\in[0, T]$ and $\varepsilon > 0$. 
Since $Y_E, Y_I, Y_Q \in W^{1,2}\bigl(0, T; L^2(\Omega)\bigr) \subset C\bigl(0, T; L^2(\Omega)\bigr)$, 
we can apply the Ascoli–Arzelà theorem together with the uniqueness of limits to conclude that
\begin{equation*}
\begin{aligned}
    & \lim _{\varepsilon \to 0} \|Y_E^{\varepsilon}(T) - Y_E(T)\|_{L^2(\Omega)} 
    \leq \lim _{\varepsilon \to 0} \max _{t \in[0, T]} \|Y_E^{\varepsilon}(t) - Y_S(t)\|_{L^2(\Omega)}=0, \\
    & \lim _{\varepsilon \to 0} \|Y_I^{\varepsilon}(T) - Y_I(T)\|_{L^2(\Omega)} 
    \leq \lim _{\varepsilon \to 0} \max _{t \in[0, T]} \|Y_I^{\varepsilon}(t) - Y_I(t) \|_{L^2(\Omega)}=0,\\
& \lim _{\varepsilon \to 0} \|Y_Q^{\varepsilon}(T) - Y_Q(T) \|_{L^2(\Omega)} 
   \leq \lim _{\varepsilon \to 0} \max _{t\in[0, T]} \|Y_Q^{\varepsilon}(t) - Y_Q(t) \|_{L^2(\Omega)}=0, 
\end{aligned}
\end{equation*}
which means that
\begin{equation}\label{E5.12}
    Y_E^{\varepsilon}(T, x) \to Y_E(T, x), \ Y_I^{\varepsilon}(T, x) \to Y_I(T, x), \ Y_Q^{\varepsilon}(T, x) \to Y_Q(T, x)  \text { in } L^1(\Omega) \ \text { as } \ \varepsilon \to 0\,.
\end{equation}
Now, combining~\eqref{E5.10} and~\eqref{E5.12} and passing to the limit as $\varepsilon \to 0$ in~\eqref{E5.9}, for any $\widetilde{u} \in \bigl(L^2(\mathcal{U})\bigr)^3$, we obtain 
\begin{equation}\label{E5.13}
\begin{aligned}
    0 \leq &\int_0^T \int_\Omega (\kappa_1 Y_E + \kappa_1 Y_I + \kappa_2 Y_Q + w_1 \widetilde{u}_1 + w_2 \widetilde{u}_2 + w_3 \widetilde{u}_3 ) \,dx dt\\
    &+ \int_\Omega \bigl( \kappa_3 Y_E(T,x) + \kappa_3 Y_I(T,x) + \kappa_4 Y_Q(T,x) + \sigma_1 \widetilde{u}_1(T,x) + \sigma_2 \widetilde{u}_2(T,x) + \sigma_3 \widetilde{u}_3(T,x) \bigr)\, dx.
\end{aligned}
\end{equation}
Multiplying the first five equations of system~\eqref{E5.13} by $Y_S^{\varepsilon}$, $Y_E^{\varepsilon}$, $Y_I^{\varepsilon}$, $Y_Q^{\varepsilon}$, and $Y_R^{\varepsilon}$, respectively, we obtain
\begin{equation}\label{E5.14}
\left\{\begin{aligned}
  \frac{\partial P_S}{\partial t} Y_S 
  &= -\lambda_S \Delta P_S Y_S + \bigl(\mu + u_1^* + \beta_1(1-u_3^*)E^* + \beta_2(1-u_3^*)I^* \bigr) P_S Y_S \\
& \qquad - \beta_1(1-u_3^*)E^* P_E Y_S - \beta_2(1-u_3^*)I^* P_I Y_S - u_1^* P_R Y_S, \\
\frac{\partial P_E}{\partial t} Y_E 
&= -\lambda_E \Delta P_E Y_E + \beta_1(1-u_3^*)S^* P_S Y_E 
+ \bigl(\delta + \mu - \beta_1(1-u_3^*)S^* \bigr) P_E Y_E\\
& \qquad - \delta P_I Y_E + \kappa_1 Y_E, \\
\frac{\partial P_I}{\partial t} Y_I 
&= -\lambda_I \Delta P_I Y_I + \beta_2(1-u_3^*)S^* P_S Y_I 
+ \bigl(\gamma + \mu - \beta_2(1-u_3^*)S^* \bigr)P_I Y_I\\
& \qquad - \gamma P_Q Y_I + \kappa_1 Y_I, \\
\frac{\partial P_Q}{\partial t} Y_Q 
&= -\lambda_Q \Delta P_Q Y_Q - \rho P_S Y_Q + (\alpha + \rho + \mu + u_2^*)P_Q Y_Q\\
& \qquad - (\alpha + u_2^*)P_R Y_Q + \kappa_2 Y_Q, \\
\frac{\partial P_R}{\partial t} Y_R &= -\lambda_R \Delta P_R Y_R + \mu P_R Y_R.
\end{aligned}\right. \quad \text{ in } \mathcal{U}, 
\end{equation}
Using the same technique, but applied to~\eqref{E5.5} and multiplying by $P_S$, $P_E$, $P_I$, $P_Q$, and $P_R$, respectively, we obtain \begin{equation}\label{E5.15}
\left\{\begin{aligned}
    \frac{\partial Y_S}{\partial t} P_S &= \lambda_S \Delta Y_S P_S - (\mu + u_1^* + \beta_1(1-u_3^*)E^* + \beta_2(1-u_3^*)I^*)Y_S P_S - \beta_1(1-u_3^*)S^* Y_E P_S  \\
    & \qquad- \beta_2(1-u_3^*)S^* Y_I P_S + \beta_1\widetilde{u}_3S^*E^*P_S + \beta_2\widetilde{u}_3S^*I^*P_S + \rho Y_Q P_S - \widetilde{u}_1S^* P_S,\\
    \frac{\partial Y_E}{\partial t} P_E &= \lambda_E \Delta Y_E P_E + \beta_1(1-u_3^*)E^* Y_S P_E 
      + \bigl[\beta_1(1-u_3^*)S^* P_E - (\delta + \mu)\bigr]Y_E P_E\\
    & \qquad- \beta_1\widetilde{u}_3S^*E^* P_E,\\
    \frac{\partial Y_I}{\partial t} P_I &= \lambda_I \Delta Y_I P_I + \beta_2(1-u_3^*)I^* Y_S P_I + \delta Y_E P_I 
      + \bigl[\beta_2(1-u_3^*)S^* - (\gamma + \mu)\bigr]Y_I  P_I\\
    & \qquad- \beta_2\widetilde{u}_3S^*I^* P_I,\\
    \frac{\partial Y_Q}{\partial t} P_Q &= \lambda_Q \Delta Y_Q P_Q + \gamma Y_I P_Q - (\alpha + \rho + \mu + u_2^*)Y_Q P_Q - \widetilde{u}_2Q^* P_Q,\\
    \frac{\partial Y_R}{\partial t} P_R &= \lambda_R \Delta Y_R P_R + \alpha Y_Q P_R + u_1^* Y_S P_R + u_2^* Y_Q P_R - \mu Y_R P_R + \widetilde{u}_1S^* P_R + \widetilde{u}_2Q^* P_R.
\end{aligned}\right. \ \text{ in } \mathcal{U}, 
\end{equation}
Adding \eqref{E5.14} and \eqref{E5.15}, we further obtain
\begin{equation}\label{E5.16}
\begin{aligned}
     \Bigl(\frac{\partial P_S}{\partial t} Y_S + \frac{\partial Y_S}{\partial t} P_S\Bigr) 
      &+ \Bigl(\frac{\partial P_E}{\partial t} Y_E + \frac{\partial Y_E}{\partial t} P_E \Bigr)
       + \Bigl(\frac{\partial P_I}{\partial t} Y_I + \frac{\partial Y_I}{\partial t} P_I \Bigr)\\
      &+ \Bigl(\frac{\partial P_Q}{\partial t} Y_Q + \frac{\partial Y_Q}{\partial t} P_Q \Bigr) 
       + \Bigl(\frac{\partial P_R}{\partial t} Y_R + \frac{\partial Y_R}{\partial t} P_R \Bigr)\\
     =& - (\lambda_S \Delta P_S Y_S + \lambda_E \Delta P_E Y_E + \lambda_I \Delta P_I Y_I + \lambda_Q \Delta P_Q Y_Q + \lambda_R \Delta P_R Y_R)\\
     &+ \lambda_S \Delta Y_S P_S + \lambda_E \Delta Y_E P_E + \lambda_I \Delta Y_I P_I + \lambda_Q \Delta Y_Q P_Q + \lambda_R \Delta Y_R P_R\\
     &+ \kappa_1 Y_E + \kappa_1 Y_I + \kappa_2 Y_Q + \beta_1\widetilde{u}_3S^*E^*P_S + \beta_2\widetilde{u}_3S^*I^*P_S - \widetilde{u}_1S^* P_S\\
     &- \beta_1\widetilde{u}_3S^*E^* P_E - \beta_2\widetilde{u}_3S^*I^* P_I - \widetilde{u}_2Q^* P_Q + \widetilde{u}_1S^* P_R + \widetilde{u}_2Q^* P_R.
\end{aligned}
\end{equation}
Combining the initial and final conditions in~\eqref{E5.5} and~\eqref{E5.13}, yields
\begin{equation*}
\begin{aligned}
&\bullet \int_0^T \int_\Omega\Bigl(\frac{\partial P_S}{\partial t} Y_S + \frac{\partial Y_S}{\partial t} P_S \Bigr)\, dx \, dt 
= \int_\Omega\bigl[P_SY_S\bigr]_0^T dx = 0,\\
&\bullet \int_0^T \int_\Omega\Bigl(\frac{\partial P_E}{\partial t} Y_E + \frac{\partial Y_E}{\partial t} P_E\Bigr) \,dx \, dt 
= \int_\Omega\bigl[P_EY_E\bigr]_0^T dx = \int_\Omega-\kappa_3 Y_E(T,x) dx,\\
&\bullet \int_0^T \int_\Omega\Bigl(\frac{\partial P_I}{\partial t} Y_I + \frac{\partial Y_I}{\partial t} P_I\Bigr) \,dx \, dt = \int_\Omega\bigl[P_IY_I\bigr]_0^T dx = \int_\Omega-\kappa_3 Y_I(T,x) dx,\\
&\bullet \int_0^T \int_\Omega\Bigl(\frac{\partial P_Q}{\partial t} Y_Q + \frac{\partial Y_Q}{\partial t} P_Q\Bigr) \,dx \ dt = \int_\Omega\bigl[P_QY_Q\bigr]_0^T dx = \int_\Omega-\kappa_4 Y_Q(T,x) dx,\\
&\bullet \int_0^T \int_\Omega\Bigl(\frac{\partial P_R}{\partial t} Y_R + \frac{\partial Y_R}{\partial t} P_R\Bigr) \,dx \, dt = \int_\Omega\bigl[P_RY_R\bigr]_0^T dx = 0.\\
\end{aligned}
\end{equation*}
Using the divergence theorem together with the homogeneous boundary conditions, we get
\begin{equation*}
\begin{aligned}
0 =& -\int_0^T \int_\Omega (\lambda_S \Delta P_S Y_S + \lambda_E \Delta P_E Y_E + \lambda_I \Delta P_I Y_I + \lambda_Q \Delta P_Q Y_Q + \lambda_R \Delta P_R Y_R )\, dx \, dt\\
&+ \int_0^T \int_\Omega (\lambda_S \Delta Y_S P_S + \lambda_E \Delta Y_E P_E + \lambda_I \Delta Y_I P_I + \lambda_Q \Delta Y_Q P_Q + \lambda_R \Delta Y_R P_R ) \,dx \, dt.
\end{aligned}
\end{equation*}
Thus, \eqref{E5.16} becomes
\begin{equation}\label{E5.17}
\begin{aligned}
\int_\Omega & \bigl(-\kappa_3 Y_E(T,x) - \kappa_3 Y_I(T,x) - \kappa_4 Y_Q(T,x) \bigr)\,dx\\
=& \int_0^T \int_\Omega \big(\kappa_1 Y_E + \kappa_1 Y_I + \kappa_2 Y_Q + \beta_1\widetilde{u}_3S^*E^*P_S + \beta_2\widetilde{u}_3S^*I^*P_S - \widetilde{u}_1S^* P_S\\
&\quad - \beta_1\widetilde{u}_3S^*E^* P_E - \beta_2\widetilde{u}_3S^*I^* P_I - \widetilde{u}_2Q^* P_Q + \widetilde{u}_1S^* P_R + \widetilde{u}_2Q^* P_R \big) \,dx dt.
\end{aligned}
\end{equation}
Substituting \eqref{E5.17} into \eqref{E5.13}, we obtain
\begin{equation*}
\begin{aligned}
\int_0^T \int_\Omega & (S^* P_S - S^* P_R + w_1) \widetilde{u}_1\, dx \, dt 
+ \int_0^T \int_\Omega (Q^* P_Q - Q^* P_R + w_2) \widetilde{u}_2 \,dx \, dt\\
&+ \int_0^T \int_\Omega \bigl(\beta_1 S^*E^* (P_E - P_S) + \beta_2S^*I^* (P_I - P_S) + w_3 \bigr) \widetilde{u}_3 \,dx \, dt\\
\geq& - \int_\Omega \sigma_1 \widetilde{u}_1(T,x) \, dx - \int_\Omega  \sigma_2 \widetilde{u}_2(T,x) \, dx - \int_\Omega \sigma_3 \widetilde{u}_3(T,x) \, dx.
\end{aligned}
\end{equation*}
Recall that $\widetilde{u} \in \bigl(L^2(\mathcal{U})\bigr)^3$ is arbitrary.
Thus, by taking $\widetilde{u} = u - u^*$ for any $u \in \mathcal{V}_{ad}$, the inequality~\eqref{E5.7} is established.

Moreover, if $\sigma_1 = \sigma_2 = \sigma_3 \equiv 0$, then from~\eqref{E5.7}, 
for any $u \in \mathcal{V}_{ad}$, we have
\begin{equation*}
\begin{aligned}
\int_0^T \int_\Omega & (S^* P_S - S^* P_R + w_1) (u_1 - u_1^*) \,dx \, dt 
 + \int_0^T \int_\Omega (Q^* P_Q - Q^* P_R + w_2) (u_2 - u_2^*) \,dx \, dt\\
&+ \int_0^T \int_\Omega \bigl(\beta_1 S^*E^* (P_E - P_S) + \beta_2S^*I^* (P_I - P_S) + w_3\bigr) (u_3 - u_3^*) \,dx \, dt \geq 0.
\end{aligned}
\end{equation*}
Owing to the arbitrariness of $u \in \mathcal{V}_{ad}$,
the above inequality is equivalent to
\begin{equation}\label{E5.18}
\begin{aligned}
&\bullet \ \int_0^T \int_\Omega [S^*P_S - S^*P_R + w_1] (u_1 - u_1^*) \, dx\,dt \geq 0,\\
&\bullet \ \int_0^T \int_\Omega [Q^*P_Q - Q^*P_R + w_2] (u_2 - u_2^*) \, dx\,dt \geq 0,\\
&\bullet \ \int_0^T \int_\Omega \bigl[\beta_1S^*E^*(P_E - P_S) + \beta_2S^*I^*(P_I - P_S) + w_3\bigr] (u_3 - u_3^*) \, dx\,dt \geq 0.
\end{aligned}
\end{equation}
We consider the first inequality in \eqref{E5.18}. 
Let $\Theta_1 = \bigl\{(t, x) \in \mathcal{U} \;|\; (S^*P_S - S^*P_R + w_1)(t,x) \leq 0\bigr\}$, 
and $\Theta_2 = \bigl\{(t, x) \in \mathcal{U} \;|\; (S^*P_S - S^*P_R + w_1)(t,x) > 0\bigr\}$.
Then, we have
\begin{equation*}
\begin{aligned}
\int_0^T &\int_{\Omega} [S^*P_S - S^*P_R + w_1] (u_1 - u_1^*)(t,x) \,dx \,dt \\
&= \int_{\Theta_1} [S^*P_S - S^*P_R + w_1] (u_1 - u_1^*)(t,x) \,dx \,dt 
+ \int_{\Theta_2} [S^*P_S - S^*P_R + w_1] (u_1 - u_1^*)(t,x) \,dx \,dt.
\end{aligned}
\end{equation*}
It is easy to show that the first inequality in~\eqref{E5.18} holds for any $u \in \mathcal{U}$ if and only if
\begin{equation*}
   \int_{\Theta_1} [S^*P_S - S^*P_R + w_1] (u_1 - u_1^*) \, dx\,dt \geq 0,
\end{equation*}
and
\begin{equation*}
   \int_{\Theta_2} [S^*P_S - S^*P_R + w_1] (u_1 - u_1^*) \, dx\,dt \geq 0,
\end{equation*}
which implies
\begin{equation*}
    u_1 - u_1^* \leq 0 \ \text{ for all } (t, x) \in \Theta_1, \ \text{ and } \ u_1 - u_1^* \geq 0 \ \text{ for all } (t, x) \in \Theta_2.
\end{equation*}
Therefore, $u_1^* \equiv 1$ on $\Theta_1$, and $u_1^* \equiv 0$ on $\Theta_2$.
We finally obtain 
\begin{equation*}
\begin{aligned}
  u_1^*(t,x) &= \begin{cases}
   1, & \text{if } (S^*P_S - S^*P_R + w_1)(t,x) \leq 0,\\
   0, & \text{otherwise}.
\end{cases}
\end{aligned}
\end{equation*}
Similarly,
\begin{equation*}
\begin{aligned}
  u_2^*(t,x) &= \begin{cases}
    1, & \text{if } (Q^*P_Q - Q^*P_R + w_2)(t,x) \leq 0,\\
    0, & \text{otherwise},
\end{cases}
\end{aligned}
\end{equation*}
and
\begin{equation*}
\begin{aligned}
  u_3^*(t,x) &= \begin{cases}
    1, & \text{if } (\beta_1S^*E^*(P_E - P_S) + \beta_2S^*I^*(P_I - P_S) + w_3)(t,x) \leq 0,\\
    0, & \text{otherwise}.
\end{cases}
\end{aligned}
\end{equation*}
This completes the proof.
\end{proof}
\end{appendices}


\section*{Declarations}

\subsection*{Data availability} 
All information analyzed or generated, which would support the results of this work are available in this article.
No data was used for the research described in the article.

\subsection*{Conflict of interest} 
The authors declare that there are no problems or conflicts 
of interest between them that may affect the study in this paper.


\bibliographystyle{unsrt}
\bibliography{References}


\end{document}